\documentclass[11pt]{article}%
\usepackage{amsfonts}
\usepackage{graphicx}
\usepackage{amsmath}
\usepackage{amssymb}
\usepackage{hyperref}%
\providecommand{\U}[1]{\protect\rule{.1in}{.1in}}
\newtheorem{theorem}{Theorem}

\newtheorem{condition}[theorem]{Condition}

\newtheorem{corollary}[theorem]{Corollary}

\newtheorem{definition}[theorem]{Definition}
\newtheorem{example}[theorem]{Example}

\newtheorem{proposition}[theorem]{Proposition}
\newtheorem{remark}[theorem]{Remark}

\newenvironment{proof}[1][Proof]{\textbf{#1.} }{\ \rule{0.5em}{0.5em}}
\begin{document}

\author{Ognyan Kounchev
\and Hermann Render
\and Tsvetomir Tsachev}
\title{Fast algorithms for interpolation with L-splines for differential operators
$L$ of order $4$ with constant coefficients$.$ \thanks{The work of H. Render
and Ts. Tsachev was funded under project KP-06-N32-8, while the work of O.
Kounchev was funded under project KP-06N42-2 with Bulgarian NSF.} }
\date{}
\maketitle

\begin{abstract}
In the classical theory of cubic interpolation splines there exists an algorithm which works 
with only $O\left(  n\right)$ arithmetic operations. Also, the smoothing cubic splines may be computed via  
the algorithm of Reinsch which  reduces their computation to interpolation cubic splines and 
also performs with $O\left(  n\right)$ arithmetic operations. 
In this paper it is shown that 
many features of the polynomial cubic spline setting carry over to the larger class
of $L$-splines where $L$ is a linear differential operator of order $4$ with
constant coefficients. Criteria are given such that the associated matrix $R$
is strictly diagonally dominant which implies the existence of a fast
algorithm for interpolation.

\end{abstract}


\begin{center}
Ognyan Kounchev \newline Institute of Mathematics and Informatics, Bulgarian
Academy of Sciences, \newline Acad. G. Bonchev str., block 8, 1113 Sofia,
Bulgaria, \texttt{kounchev@math.bas.bg} 

\bigskip
Hermann Render \newline School of Mathematics and Statistics, University
College Dublin, \newline Belfield, Dublin 4, Ireland,
\texttt{hermann.render@ucd.ie} 

\bigskip
Tsvetomir Tsachev \newline Institute of Mathematics and Informatics, Bulgarian
Academy of Sciences, \newline Acad. G. Bonchev str., block 8, 1113 Sofia,
Bulgaria, \texttt{tsachev@math.bas.bg}
\end{center}

\section{Introduction}

Interpolation and smoothing with polynomial splines is an important technique
in nonparametric regression. It is outlined in \cite{RaHeSi97} that smoothing
with more general types of splines allows to include additional prior
knowledge for estimating the data and to retain several aspects of the data.
For example, for the estimation of the gross domestic product data it is
assumed that the model curve is a linear combination of the functions
\begin{equation}
1,\exp\left(  \gamma t\right)  ,\sin\omega t,\cos\omega t, \label{eqExx1a}%
\end{equation}
while for melanoma data the model curve is linear combination of
\begin{equation}
1,t,\sin\omega t,\cos\omega t. \label{eqExx2a}%
\end{equation}
In \cite{GreenSilverman}, \cite{Gu}, \cite{RamsaySilverman} the reader can
find an efficient algorithm for $L$-spline smoothing based on reproduction
kernel techniques associated to the differential operator. On the other hand
it is a natural question whether the ideas and methods for cubic splines (see
\cite{Rein67}) can be modified to obtain efficient algorithms for interpolation
and smoothing for special classes of $L$-splines. In our recent paper
\cite{KounchevRenderTsachevBIT} we have shown this for the case where the
differential operator $L$ is of the form
\[
L=\left(  \frac{d}{dt}-\xi^{2}\right)  ^{2}%
\]
and $\xi$ is a real number. Our motivation for the analysis of these
univariate problems stems from the fact that these differential operators play
an important role in the study of polysplines of order $2$ on strips (see
\cite{kounchevrenderappr}, \cite{kounchevrenderpams}, \cite{KoReCM},
\cite{kounchevrenderJAT}). For the study of polysplines of order $2$ on
annular regions in $\mathbb{R}^{d}$ differential operators of the form
\[
L=\left(  \frac{d}{dt}-k\right)  \left(  \frac{d}{dt}-k-2\right)  \left(
\frac{d}{dt}+k+2-d\right)  \left(  \frac{d}{dt}+k+4-d\right)
\]
are needed where $k\in\mathbb{N}_{0}$. This has motivated us to study the
question whether the results for classical cubic splines can be extended to
any linear differential operators $L$ with constant coefficients of order $4$.
Note that the basis functions in examples (\ref{eqExx1a}) and (\ref{eqExx2a})
are also solutions of differential operators of order $4.$

Let us recall that a linear differential operator of order $N+1$ is of the
form
\begin{equation}
L=L_{\left(  \lambda_{0},...,\lambda_{N}\right)  }=\prod_{j=0}^{N}\left(
\frac{d}{dx}-\lambda_{j}\right)  . \label{neuLa}%
\end{equation}
where $\lambda_{0},...,\lambda_{N}$ are arbitrary complex numbers, and we say
that a function $g:\left[  a,b\right]  \rightarrow\mathbb{C}$ is an $L$-spline
of order $N+1$ with knots $t_{1}<...<t_{n}$ if $g$ is $N-1$ times continuously
differentiable such that the restriction of $g$ to each interval $\left(
t_{j},t_{j+1}\right)  $ is a solution of the equation $L_{N}\left(  g\right)
=0$, see e.g. \cite{Sc81}. Moreover we define
\begin{equation}
E\left(  \lambda_{0},...,\lambda_{N}\right)  :=\left\{  f\in C^{N}\left(
\mathbb{R},\mathbb{C}\right)  :L_{N}f=0\right\}
\end{equation}
where $C^{N}\left(  \mathbb{R},\mathbb{C}\right)  $ is the space of all
$N$-times continuously differentiable complex-valued functions $f:\mathbb{R}%
\rightarrow\mathbb{C}$. The elements in $E_{\left(  \lambda_{0},\ldots
,\lambda_{N}\right)  }$ are called \emph{exponential polynomials} or
$L$\emph{-polynomials, }and $\lambda_{0},\ldots,\lambda_{N}$ are called
\emph{exponents} or \emph{frequencies} (see e.g.\ Chapter 3 in \cite{BeGa95}).

It is matter of fact that many important formulas occurring in the analysis of
$L$-splines can be expressed in terms of the so-called fundamental function
$\Phi_{\Lambda_{N}}$ -- and that this is a great advantage has been
demonstrated in the case of error estimates for interpolation of $L$-splines
of order $4$ in \cite{KounchevRender2021JCAM}. Let us recall that for every
$\Lambda_{N}=\left(  \lambda_{0},\ldots,\lambda_{N}\right)  \in\mathbb{C}%
^{N+1}$ there exists a unique solution $\Phi_{\Lambda_{N}}\in E_{\left(
\lambda_{0},\ldots,\lambda_{N}\right)  }$ to the Cauchy problem
\[
\Phi_{\Lambda_{N}}\left(  0\right)  =\ldots=\Phi_{\Lambda_{N}}^{\left(
N-1\right)  }\left(  0\right)  =0\text{ and }\Phi_{\Lambda_{N}}^{\left(
N\right)  }\left(  0\right)  =1,
\]
and $\Phi_{\Lambda_{N}}$ is called the \emph{fundamental function} in
$E_{\left(  \lambda_{0},\ldots,\lambda_{N}\right)  }$ (see Proposition 13.12
in \cite{okbook}, or Chapter $9$ in \cite{Sc81}, or \ \cite{micchelli}). Since
$\Phi_{\Lambda_{N}}$ is an analytic function there exists $\delta>0$ such
that
\[
\Phi_{\Lambda_{N}}\left(  x\right)  \neq0\text{ for all }x\in\left(
0,\delta\right)  .
\]
Throughout this paper we make the following assumption:

\begin{condition}
\label{ConditionSTAR} For a given linear differential operator $L_{\left(
\lambda_{0},...,\lambda_{3}\right)  }$ assume that the nodes $t_{_{1}%
}<...<t_{_{n}}$ in the definition of an $L$-spline satisfy
\[
t_{j+1}-t_{j}<\delta\quad\quad\text{ for all }j=0,...,n-1
\]
where
\[
\delta=\min\left\{  \frac{2\pi}{\left\vert \operatorname{Im}\left(
\lambda_{1}-\lambda_{0}\right)  \right\vert },\frac{2\pi}{\left\vert
\operatorname{Im}\left(  \lambda_{3}-\lambda_{2}\right)  \right\vert
}\right\}  .
\]

\end{condition}

It is easy to see (cf. Proposition \ref{PropImaginary}) that Condition
\ref{ConditionSTAR} implies that
\begin{equation}
\Phi_{\left(  \lambda_{_{0}},\lambda_{_{1}}\right)  }\left(  t\right)
\neq0\text{ and }\Phi_{\left(  \lambda_{_{2}},\lambda_{_{3}}\right)  }\left(
t\right)  \neq0 \label{PhiCondb}%
\end{equation}
for all $t\in\left(  0,\delta\right)  .$ Note that for real $\lambda
_{0},...,\lambda_{3}$ the imaginary part of $\lambda_{1}-\lambda_{0}$ and
$\lambda_{3}-\lambda_{0}$ is zero and we may choose $\delta=\infty$ in
Condition \ref{ConditionSTAR}.

Now let us recall some basic facts about interpolation splines in the
classical polynomial setting: for given real numbers $t_{1}<...<t_{n}$ with
$n>2$ there exists a symmetric positive definite $\left(  n-2\right)
\times\left(  n-2\right)  $ matrix $R$ and a $n\times\left(  n-2\right)  $
matrix $Q$ (we follow here the notation given in \cite{GreenSilverman}) such
that
\begin{equation}
Q^{T}\mathbf{g}=R\gamma\label{eq1}%
\end{equation}
for any natural (i.e. with vanishing second derivative at $t_{1}$ and at
$t_{n}$) cubic spline $g$ with knots $t_{1}<...<t_{n}$ where \
\begin{align}
\mathbf{g}^{T}  &  =\left(  g\left(  t_{1}\right)  ,g\left(  t_{2}\right)
,...,g\left(  t_{n}\right)  \right) \label{eq2}\\
\gamma^{T}  &  =\left(  g^{\prime\prime}\left(  t_{2}\right)  ,g^{\prime
\prime}\left(  t_{3}\right)  ,...,g^{\prime\prime}\left(  t_{n-1}\right)
\right)  . \label{eq3}%
\end{align}
The matrix $R$ can be described explicitly: it has tridiagonal form and for
the diagonal and the two off-diagonals we have
\[
R_{jj}=\frac{1}{3}\left(  h_{j-1}+h_{j}\right)  \text{ and }R_{j,j+1}%
=R_{j+1,j}=\frac{1}{6}h_{j}%
\]
where $h_{j}=t_{j+1}-t_{j}$. The algorithm for the interpolation splines and
the Reinsch algorithm for smoothing splines are based on the matrices $R$ and
$Q$ allowing the computation of the interpolating \emph{natural} spline in
$O\left(  n\right)  $ arithmetic operations. In the context of $L$-splines we define

\begin{definition}
An $L-$spline for $\Lambda_{3}=\left(  \lambda_{0},\lambda_{1},\lambda
_{2},\lambda_{3}\right)  $ is called \textbf{natural} if it satisfies the
conditions
\[
L_{(\lambda_{0},\lambda_{1})}g\left(  t_{1}\right)  =L_{(\lambda_{0}%
,\lambda_{1})}g\left(  t_{n}\right)  =0.
\]

\end{definition}

Note that the definition of an $L$-spline does not depend on the order of
$\lambda_{0},...,\lambda_{3}$ in the vector $\Lambda_{3}=\left(  \lambda
_{0},\lambda_{1},\lambda_{2},\lambda_{3}\right)  .$ However, the definition of
a \emph{natural} spline is related to the first two entries $\lambda
_{0},\lambda_{1}$ in the vector $\left(  \lambda_{0},\lambda_{1},\lambda
_{2},\lambda_{3}\right)  $ (we may interchange $\lambda_{0}$ and $\lambda
_{1}).$

Now let us outline the results in this paper: In Section 2 we shall recall
several properties of the fundamental function $\Phi_{\Lambda_{N}}.$ In
Section 3 we shall show that under the general Condition \ref{ConditionSTAR}
one can define a $\left(  n-2\right)  \times\left(  n-2\right)  $ matrix $R$
and a $n\times\left(  n-2\right)  $ matrix $Q$ such that the relation
(\ref{eq1}) holds for the vectors $\mathbf{g},$ and $\gamma$ defined by \
\begin{align}
\mathbf{g}^{T}  &  =\left(  g\left(  t_{1}\right)  ,g\left(  t_{2}\right)
,...,g\left(  t_{n}\right)  \right) \label{g}\\
\gamma^{T}  &  =\left(  L_{\left(  \lambda_{0},\lambda_{1}\right)  }g\left(
t_{2}\right)  ,L_{\left(  \lambda_{0},\lambda_{1}\right)  }g\left(
t_{3}\right)  ,...,L_{\left(  \lambda_{0},\lambda_{1}\right)  }g\left(
t_{n-1}\right)  \right)  \label{gamma}%
\end{align}
where $g$ is a natural spline -- hence $\gamma_{j}:=L_{\left(  \lambda
_{0},\lambda_{1}\right)  }g\left(  t_{j}\right)  $ is equal to zero for $j=1$
and $j=n.$ The entries of the matrix $R$ and $Q$ will depend on the step size
$h_{j}=t_{j+1}-t_{j}$ and on the frequencies $\lambda_{0},...,\lambda_{3}$ --
and on the order of $\lambda_{0},..,\lambda_{3}$ in the vector $\left(
\lambda_{0},\lambda_{1},\lambda_{2},\lambda_{3}\right)  .$ More precisely, we
will show that the matrix $R$ is tridiagonal and that the diagonal and the
first off-diagonal entries are given by
\begin{align}
R_{j,j}  &  =\rho\left(  t_{j}-t_{j-1}\right)  -\rho\left(  -\left(
t_{j+1}-t_{j}\right)  \right) \label{eqRjj}\\
R_{j,j+1}  &  =\sigma\left(  t_{j+1}-t_{j}\right)  \quad\text{ and }\quad
R_{j+1,j}=-\sigma\left(  -\left(  t_{j+1}-t_{j}\right)  \right)
\label{eqRjj1}%
\end{align}
where $\rho$ and $\sigma$ are functions defined by
\begin{align}
\rho\left(  x\right)   &  =\frac{\Phi_{\left(  \lambda_{0},...,\lambda
_{3}\right)  }^{\prime}\left(  x\right)  \Phi_{\left(  \lambda_{0},\lambda
_{1}\right)  }\left(  x\right)  -\Phi_{\left(  \lambda_{0},...\lambda
_{3}\right)  }\left(  x\right)  \Phi_{\left(  \lambda_{0},\lambda_{1}\right)
}^{\prime}\left(  x\right)  }{\Phi_{\left(  \lambda_{0},\lambda_{1}\right)
}\left(  x\right)  \Phi_{\left(  \lambda_{2},\lambda_{3}\right)  }\left(
x\right)  },\label{eqDefw}\\
\sigma\left(  x\right)   &  =\frac{\Phi_{\left(  \lambda_{0},...\lambda
_{3}\right)  }\left(  x\right)  }{\Phi_{\left(  \lambda_{0},\lambda
_{1}\right)  }\left(  x\right)  \Phi_{\left(  \lambda_{2},\lambda_{3}\right)
}\left(  x\right)  }. \label{eqDefr}%
\end{align}
Let us emphasize that a description of the matrix entries $\rho\left(
x\right)  $ and $\sigma\left(  x\right)  $ in terms of the fundamental
function $\Phi_{\Lambda_{N}}$ is often more useful than an explicit formula in
terms of the exponential functions $e^{\lambda_{0}x},...,e^{\lambda_{N}x}$. In
a similar way the matrix $Q$ can be defined.

In Section 4 we shall provide a simple sufficient criterion which guarantees
that the matrix $R$ is strictly diagonally dominant. In the case that $\left(
\lambda_{0},\lambda_{1}\right)  \in\mathbb{C}^{2}$ and $\left(  \lambda
_{2},\lambda_{3}\right)  \in\mathbb{C}^{2}$ are conjugation invariant we shall
show that for any $\varepsilon>0$ there exists $\delta>0$ such that the
estimate
\[
\left\vert R_{j,j-1}\right\vert +\left\vert R_{j,j+1}\right\vert \leq\left(
\frac{1}{2}+\varepsilon\right)  \left\vert R_{jj}\right\vert
\]
holds for all $t_{1}<....<t_{n}$ such that $t_{j+1}-t_{j}<\delta$ for
$j=1,...,n-1.$ Thus we obtain strict diagonal dominance of the matrix $R$ if
we make the step size $t_{j+1}-t_{j}$ sufficiently small.

In Section 5 we will show that for special vectors $\Lambda_{3}$ stronger
results hold: for $\Lambda_{3}=\left(  -a,a,-b,b\right)  \in\mathbb{R}^{4}$
with $a,b\geq0$ we show the estimate
\[
R_{j,j-1}+R_{j,j+1}\leq\frac{1}{2}R_{jj}.
\]
For the example in (\ref{eqExx2a}) we shall discuss in detail the question of
diagonal dominance of the matrix $R.$ Examples will show that an estimate of
the form
\begin{equation}
R_{j,j-1}+R_{j,j+1}\leq C_{a,b}R_{jj}\text{ and }0<C_{a,b}<1
\label{eqRniceestimate}%
\end{equation}
does not hold for general real frequencies $\lambda_{0},...,\lambda_{3}.$

In Section 6 we show that the entries in (\ref{eqRjj}) and (\ref{eqRjj1}) are
positive numbers provided that $\lambda_{0},...,\lambda_{3}$ are real numbers.
In combination with the fact that a strictly diagonally dominant matrix with
non-negative entries is positive definite, one deduce from these results
sufficient criteria for the positive-definiteness of the matrix $R.$

It is well known that in the polynomial case the matrix R is symmetric. It is
interesting to provide necessary and sufficient conditions that $R$ is
symmetric. In Section 7 we will show that the matrix $R$ is symmetric if and
only if the vector $\left(  \lambda_{0},...,\lambda_{3}\right)  $ is
symmetric, and this is equivalent to say that the function $\sigma\left(
x\right)  $ defined in (\ref{eqDefr}) is odd. On the other hand, we shall show
that the function $\rho\left(  x\right)  $ defined in (\ref{eqDefw}) is odd if
and only if $\lambda_{0}+\lambda_{1}=\lambda_{2}+\lambda_{3}.$

$\ $

\section{The fundamental function}

We assume that we are given a vector $\Lambda\in\mathbb{C}^{^{N+1}},$ with the
first coordinate having index $0$, namely
\[
\Lambda=\Lambda_{_{N}}=\left(  \lambda_{_{0}},...,\lambda_{_{N}}\right)  .
\]
In the case of \emph{pairwise different} $\lambda_{j},j=0,\ldots,N,$ the space
$E_{\left(  \lambda_{0},\ldots,\lambda_{N}\right)  }$ is the linear span
generated by the functions $e^{\lambda_{0}x},e^{\lambda_{1}x},\ldots
,e^{\lambda_{N}x}.$ In the case when $\lambda_{j}$ occurs $m_{j}$ times in
$\Lambda_{N}=\left(  \lambda_{0},\ldots,\lambda_{N}\right)  ,$ a basis of the
space $E_{\left(  \lambda_{0},\ldots,\lambda_{N}\right)  }$ is given by the
linearly independent functions
\begin{equation}
x^{s}e^{\lambda_{j}x}\qquad\text{for }s=0,1,\ldots,m_{j}-1.
\end{equation}
In the case that $\lambda_{0}=\cdots=\lambda_{N}$ $=0,$ the space $E_{\left(
\lambda_{0},\ldots,\lambda_{N}\right)  }$ is just the space of all polynomials
of degree at most $N,$ and we shall refer to this as the \emph{polynomial
case}.

We say that the space $E_{\left(  \lambda_{_{0}},...,\lambda_{_{N}}\right)  }$
is \emph{closed under complex conjugation}, if for $f\in$ $E_{\left(
\lambda_{_{0}},...,\lambda_{_{N}}\right)  }$ the complex conjugate function
$\overline{f}$ is again in $E_{\left(  \lambda_{0},...,\lambda_{N}\right)  }.$
It is easy to see that for complex numbers $\lambda_{_{0}},...,\lambda_{_{N}}$
the space $E_{\left(  \lambda_{0},...,\lambda_{N}\right)  }$ is closed under
complex conjugation if and only if there exists a permutation $p$ of the
indices $\left\{  0,...,N\right\}  $ such that $\overline{\lambda_{_{j}}%
}=\lambda_{_{p\left(  j\right)  }}$ for $j=0,...,N.$ In other words,
$E_{\left(  \lambda_{_{0}},...,\lambda_{_{N}}\right)  }$ is closed under
complex conjugation if and only if the vector $\Lambda_{_{N}}=\left(
\lambda_{_{0}},...,\lambda_{_{N}}\right)  $ is equal up to reordering to the
conjugate vector $\overline{\Lambda_{_{N}}}$.

We say that the vector $\Lambda_{N}$ is \emph{symmetric} if there exists a
permutation $p$ of the set $\left\{  0,...,N\right\}  $ such that
$-\lambda_{_{j}}=\lambda_{_{p(j)}}$ for $j=0,...,N,$ or, symbolically,
$-\Lambda_{N}=\Lambda_{N}.$ If $\Lambda_{N}$ is symmetric then
\[
\Phi_{\Lambda_{N}}(-x)=(-1)^{N}\Phi_{\Lambda_{N}}(x).
\]
The formula shows that for odd $N$ the complex-valued function $\Phi
_{\Lambda_{N}}$ is odd, and for even $N$ the function $\Phi_{\Lambda_{N}}$ is even.

Here are two examples for the fundamental function.

\begin{example}
In (\ref{eqExx1a}) we have the frequencies $\Lambda_{3}=(0,\gamma
,i\omega,-i\omega)$ with $\gamma\neq0$ and $\omega\neq0.$ Then the fundamental
function is given by
\[
\Phi_{\Lambda_{3}}\left(  t\right)  =\frac{\gamma^{2}\left(  \cos\omega
t-1\right)  -\omega\gamma\sin\omega t+\omega^{2}\left(  e^{\gamma t}-1\right)
}{\omega^{2}\gamma\left(  \omega^{2}+\gamma^{2}\right)  }%
\]

\end{example}

\begin{example}
In (\ref{eqExx2a}) we have the frequencies $\Lambda_{_{3}}=\left(
0,0,i\omega,-i\omega\right)  $ and the fundamental function is given by%
\[
\Phi_{\Lambda_{_{3}}}\left(  t\right)  =\frac{\omega t-\sin\omega t}%
{\omega^{3}}.
\]

\end{example}

An explicit formula for $\Phi_{\Lambda_{_{N}}}$ is%

\begin{equation}
\Phi_{\Lambda_{_{N}}}\left(  x\right)  =\frac{1}{2\pi i}\int_{\Gamma_{r}}%
\frac{e^{xz}}{\left(  z-\lambda_{_{0}}\right)  \cdots\left(  z-\lambda_{_{N}%
}\right)  }dz, \label{defPhi}%
\end{equation}
where $\Gamma_{r}$ is the path in the complex plane defined by $\Gamma
_{r}\left(  t\right)  =re^{it}$, $t\in\left[  0,2\pi\right]  $, surrounding
all the complex numbers $\lambda_{_{0}},\ldots,\lambda_{_{N}}$ (see
Proposition 13.12 in \cite{okbook}). \newline

The fundamental function can be seen as analogue of the power function $x^{N}$
in the space $E_{\left(  \lambda_{_{0}},\ldots,\lambda_{_{N}}\right)  }$: In
the polynomial case, i.e., $\lambda_{_{0}}=\ldots=\lambda_{_{N}}=0$, the
fundamental function is
\begin{equation}
\Phi_{\text{pol,}N}\left(  x\right)  =\frac{1}{N!}x^{N}. \label{Fundpol}%
\end{equation}

The next three properties follow directly from (\ref{defPhi}):
\begin{equation}
(\frac{d}{dx}-\lambda_{_{N+1}})\Phi_{\left(  \lambda_{_{0}},\ldots
,\lambda_{_{N+1}}\right)  }(x)=\Phi_{\left(  \lambda_{_{0}},\ldots
,\lambda_{_{N}}\right)  }(x), \label{rec0}%
\end{equation}
\begin{equation}
\Phi_{\Lambda_{_{N}}}\left(  -x\right)  =\left(  -1\right)  ^{^{N}}%
\Phi_{-\Lambda_{N}}\left(  x\right)  \label{neg_lambdas}%
\end{equation}
where $-\Lambda_{_{N}}$ is the vector $\left(  -\lambda_{_{0}},...,-\lambda
_{_{N}}\right)  $ and
\begin{equation}
\Phi_{\overline{\Lambda_{N}}}\left(  x\right)  =\overline{\Phi_{\Lambda_{_{N}%
}}\left(  x\right)  } \label{conjug_lambdas}%
\end{equation}
where $\overline{\Lambda_{_{N}}}$ is the vector $\left(  \overline
{\lambda_{_{0}}},...,\overline{\lambda_{_{N}}}\right)  $.

From (\ref{conjug_lambdas}) we directly obtain

\begin{proposition}
If $\Lambda_{_{N}} = \overline{\Lambda_{_{N}}}$ then $\Phi_{\Lambda_{_{N}}}
\left(  t \right)  $ is a real-valued function.
\end{proposition}

\begin{proposition}
If $\lambda_{_{0}},...,\lambda_{_{N}}$ are real numbers then $\Phi_{\left(
\lambda_{_{0}},...,\lambda_{_{N}}\right)  } \left(  t \right)  > 0$ for all
$t>0$.
\end{proposition}

\begin{proof}
It is well known that $E_{\left(  \lambda_{_{0}},...,\lambda_{_{N}} \right)
}$ is an extended Chebyshev system over the real line, so each function $f\in
E_{\left(  \lambda_{_{0}},...,\lambda_{_{N}} \right)  }$ has at most $N$ zeros
(including multiplicities). Since $\Phi_{\left(  \lambda_{_{0}},...,\lambda
_{_{N}}\right)  } \left(  t \right)  $ has at most $N$ zeros we infer that
$\Phi_{\left(  \lambda_{_{0}},...,\lambda_{_{N}} \right)  } \left(  t \right)
\neq0$ for all $t \neq0.$ The result follows since $\Phi_{\left(
\lambda_{_{0}},...,\lambda_{_{N}} \right)  }^{\left(  N \right)  } \left(  0
\right)  = 1$ and $\Phi_{\left(  \lambda_{_{0}},...,\lambda_{_{N}} \right)
}^{\left(  j \right)  } \left(  0 \right)  = 0$ for all $j=0,...,N-1$.
\end{proof}

In the following we describe the fundamental function $\Phi_{\left(
\lambda_{_{0}},\lambda_{_{1}}\right)  }$: if $\lambda_{_{0}}=\lambda_{_{1}%
}=0,$ or if $\lambda_{_{0}}=\lambda_{_{1}}\neq0$ then
\[
\Phi_{(0,0)}(x)=x=e^{(\lambda_{_{0}}+\lambda_{_{1}})x/2}x\text{ and }%
\Phi_{(\lambda_{_{0}},\lambda_{_{1}})}(x)=e^{(\lambda_{_{0}}+\lambda_{_{1}%
})x/2}x.
\]
In the case $\lambda_{_{0}}\neq\lambda_{_{1}}$ we have
\begin{equation}
\Phi_{(\lambda_{_{0}},\lambda_{_{1}})}\left(  x\right)  =\frac{e^{\lambda
_{1}x}-e^{\lambda_{0}x}}{\lambda_{_{1}}-\lambda_{_{0}}}=e^{(\lambda_{_{0}%
}+\lambda_{_{1}})x/2}\frac{e^{(\lambda_{_{1}}-\lambda_{_{0}})x/2}%
-e^{-(\lambda_{_{1}}-\lambda_{_{0}})x/2}}{\lambda_{_{1}}-\lambda_{_{0}}}.
\label{eqdegPhitwo}%
\end{equation}
In any of these cases there exists an odd function $\psi_{(\lambda_{_{0}%
},\lambda_{_{1}})}(x)$ such that
\begin{equation}
\Phi_{(\lambda_{_{0}},\lambda_{_{1}})}(x)=e^{^{\left(  \lambda_{_{0}}%
+\lambda_{_{1}}\hspace*{-0.05cm}\right)  x/2}}\psi_{(\lambda_{_{0}}%
,\lambda_{_{1}})}(x). \label{eqPsiEven}%
\end{equation}
Hence for all complex numbers $\lambda_{_{0}},\lambda_{_{1}}$ we have
established the following useful formula
\begin{equation}
\Phi_{(\lambda_{_{0}},\lambda_{_{1}})}(-x)=-e^{^{\left(  \lambda_{_{0}%
}+\lambda_{_{1}}\hspace*{-0.05cm}\right)  \hspace*{0.05cm}x}}\Phi
_{(\lambda_{_{0}},\lambda_{_{1}})}(x). \label{eqsym}%
\end{equation}
We deduce from this the following result which will be needed later:

\begin{proposition}
\label{PropSym}Let $\lambda_{_{0}},\lambda_{_{1}},\lambda_{_{2}}$ and
$\lambda_{_{3}}$ be complex numbers. Then
\begin{equation}
\Phi_{\left(  \lambda_{_{0}},\lambda_{_{1}}\right)  }(-x)\Phi_{\left(
\lambda_{_{2}},\lambda_{_{3}}\right)  }(-x)=e^{-(\lambda_{_{0}}+\lambda_{_{1}%
}+\lambda_{_{2}}+\lambda_{_{3}})x}\Phi_{\left(  \lambda_{_{0}},\lambda_{_{1}%
}\right)  }(x)\Phi_{\left(  \lambda_{_{2}},\lambda_{_{3}}\right)  }(x).
\label{eqdouble}%
\end{equation}
The function $\Phi_{\left(  \lambda_{_{0}},\lambda_{_{1}}\right)  }\left(
x\right)  \Phi_{\left(  \lambda_{_{2}},\lambda_{_{3}}\right)  }\left(
x\right)  $ is even if and only if $\lambda_{_{0}}+\cdots+\lambda_{_{3}}=0.$
\end{proposition}

\begin{proof}
Applying (\ref{eqsym}) to $\left(  \lambda_{_{0}},\lambda_{_{1}}\right)  $ and
$\left(  \lambda_{_{2}},\lambda_{_{3}}\right)  $ gives (\ref{eqdouble}). If
$\lambda_{_{0}}+\cdots+\lambda_{_{3}}=0$ then by (\ref{eqdouble}) the
function$\Phi_{\left(  \lambda_{_{0}},\lambda_{_{1}}\right)  }\left(
x\right)  \Phi_{\left(  \lambda_{_{2}},\lambda_{_{3}}\right)  }\left(
x\right)  $ is even. Conversely, if $\Phi_{\left(  \lambda_{_{0}}%
,\lambda_{_{1}}\right)  }\left(  x\right)  \Phi_{\left(  \lambda_{_{2}%
},\lambda_{_{3}}\right)  }\left(  x\right)  $ is even then by (\ref{eqdouble})
we have $e^{-\left(  \lambda_{0}+\lambda_{1}+\lambda_{2}+\lambda_{3}\right)
x}=1$ for all $x$ and this implies $\lambda_{0}+\cdots+\lambda_{3}=0.$
\end{proof}

\begin{proposition}
\label{PropImaginary}Let $\lambda_{0},\lambda_{1}$ be complex numbers and
define
\[
\delta=\frac{2\pi}{\left\vert \operatorname{Im}\left(  \lambda_{1}-\lambda
_{0}\right)  \right\vert }.
\]
Then $\Phi_{\left(  \lambda_{_{0}},\lambda_{_{1}}\right)  }\left(  t\right)
\neq0$ for all $t\in\left(  0,\delta\right)  .$
\end{proposition}

\begin{proof}
Assume at first that $\lambda_{1}=\lambda_{0}.$ Then $\delta=\infty$ (by
interpreting $\frac{2\pi}{0}=\infty)$. On the other hand, $\Phi_{(\lambda
_{_{0}},\lambda_{_{1}})}(x)=e^{(\lambda_{_{0}}+\lambda_{_{1}})x/2}x\neq0$ for
all $x\neq0,$ so the statement is true in this case. In the case $\lambda
_{1}\neq\lambda_{0}$ we see that $\Phi_{(\lambda_{_{0}},\lambda_{_{1}})}(x)=0$
implies that $e^{\left(  \lambda_{1}-\lambda_{0}\right)  x}=1$. On the other
hand, we see that for any positive $x<\delta=\frac{2\pi}{\left\vert
\operatorname{Im}\left(  \lambda_{1}-\lambda_{0}\right)  \right\vert }$ we
have $e^{\left(  \lambda_{1}-\lambda_{0}\right)  x}\neq1.$
\end{proof}

\section{The matrix $R$}

For complex numbers $\lambda_{_{0}},...,\lambda_{_{3}}$ we define the linear
differential operators
\begin{align*}
L_{_{1}}  &  =\left(  \frac{d}{dt}-\lambda_{_{0}}\right)  \left(  \frac{d}%
{dt}-\lambda_{_{1}}\right)  ,\quad\quad L_{_{2}}=\left(  \frac{d}{dt}%
-\lambda_{_{2}}\right)  \left(  \frac{d}{dt}-\lambda_{_{3}}\right) \\
L_{3}  &  =L_{1}L_{2}%
\end{align*}

Now we define solutions $A_{_{j}}^{^{\left[  2 \right]  }}( t ) $ and
$B_{_{j}}^{^{\left[  2 \right]  }}( t ) $ of the equation $L_{_{1}}u=0$ in
$[t_{_{j}}, t_{_{j+1}}]$ by%

\begin{equation}
A_{_{j}}^{^{\left[  2\right]  }}(t)=\frac{\Phi_{\left(  \lambda_{_{0}}%
,\lambda_{_{1}}\right)  }(t-t_{_{j+1}})}{\Phi_{\left(  \lambda_{_{0}}%
,\lambda_{_{1}}\right)  }(t_{_{j}}-t_{_{j+1}})}\text{ and }B_{_{j}}^{^{\left[
2\right]  }}(t)=\frac{\Phi_{\left(  \lambda_{_{0}},\lambda_{_{1}}\right)
}(t-t_{_{j}})}{\Phi_{\left(  \lambda_{_{0}},\lambda_{_{1}}\right)  }%
(t_{_{j+1}}-t_{_{j}})}. \label{eqdefAB}%
\end{equation}
where we use our general Condition \ref{ConditionSTAR} which implies that for
the nodes $t_{_{1}}<...<t_{_{n}}$ in the definition of an $L$-spline we have
\begin{equation}
\Phi_{\left(  \lambda_{_{0}},\lambda_{_{1}}\right)  }(t_{_{j+1}}-t_{_{j}}%
)\neq0\text{ and }\Phi_{\left(  \lambda_{_{2}},\lambda_{_{3}}\right)
}(t_{_{j+1}}-t_{_{j}})\neq0 \label{PhiCond}%
\end{equation}
for $j=1,...,n-1.$ We see that
\[
A_{_{j}}^{^{\left[  2\right]  }}(t_{_{j}})=1,\quad A_{_{j}}^{^{\left[
2\right]  }}(t_{_{j+1}})=0,\text{ and }B_{_{j}}^{^{\left[  2\right]  }%
}(t_{_{j}})=0,\quad B_{_{j}}^{^{\left[  2\right]  }}(t_{_{j+1}})=1.
\]

\begin{proposition}
\label{PropA1}Let $\lambda_{_{0}},...,\lambda_{_{3}}$ be complex numbers and
$t_{_{1}}<....<t_{_{n}}$ such that (\ref{PhiCond}) holds. Then there exist
unique function $A_{_{j}}^{^{\left[  1 \right]  }}( \cdot) $ and unique
function $B_{_{j}}^{^{\left[  1 \right]  }}( \cdot) $ in $E_{( \lambda_{_{0}%
},...,\lambda_{_{3}} )}$ such that $L_{_{2}} \left(  L_{_{1}} \left(  A_{_{j}%
}^{^{\left[  1 \right]  }}( t ) \right)  \right)  \equiv L_{_{2}} \left(
L_{_{1}} \left(  B_{_{j}}^{^{\left[  1 \right]  }}( t ) \right)  \right)
\equiv0$ in $[t_{_{j}}, t_{_{j+1}}]$ and%

\begin{align}
A_{_{j}}^{^{\left[  1 \right]  }}( t_{_{j}} ) \hspace*{-0.1cm}  &  =
\hspace*{-0.1cm} A_{_{j}}^{^{\left[  1 \right]  }} ( t_{_{j+1}} ) = 0 \text{
and }L_{_{1}}A_{_{j}}^{^{\left[  1 \right]  }} ( t_{_{j}} ) = 1; \; L_{_{1}%
}A_{_{j}}^{^{\left[  1\right]  }}( t_{_{j+1}} ) = 0,\label{A1properties}\\
B_{_{j}}^{^{\left[  1 \right]  }}( t_{_{j}} ) \hspace*{-0.1cm}  &  =
\hspace*{-0.1cm} B_{_{j}}^{^{\left[  1 \right]  }} ( t_{_{j+1}} ) = 0\text{
and }L_{_{1}}B_{_{j}}^{^{\left[  1\right]  }} ( t_{_{j}} ) = 0; \; L_{_{1}%
}B_{_{j}}^{^{\left[  1 \right]  }}( t_{_{j+1}} ) = 1. \label{B1properties}%
\end{align}

\end{proposition}

\begin{proof}
We define a function $A_{j}(\cdot)$ in $E_{(\lambda_{_{0}},...,\lambda_{_{3}%
})}$ by
\[
A_{_{j}}(t):=\Phi_{\left(  \lambda_{_{0}},...,\lambda_{_{3}}\right)
}(t-t_{_{j+1}})-A_{_{j}}^{^{\left[  2\right]  }}(t)\cdot\Phi_{\left(
\lambda_{_{0}},...,\lambda_{_{3}}\right)  }(t_{_{j}}-t_{_{j+1}}).
\]
Then it is straightforward to check $A_{_{j}}(t_{_{j+1}})=A_{_{j}}(t_{_{j}%
})=0$. Since $L_{_{1}}(A_{_{j}}(t))=\Phi_{(\lambda_{_{2}},\lambda_{_{3}}%
)}(t-t_{_{j+1}})$ we infer that $L_{_{1}}(A_{_{j}}(t_{_{j+1}}))=0$ Further
$L_{_{1}}(A_{_{j}}(t_{_{j}}))=\Phi_{(\lambda_{_{2}},\lambda_{_{3}})}(t_{_{j}%
}-t_{_{j+1}})$ and we see that
\[
A_{_{j}}^{^{\left[  1\right]  }}(t)=\frac{\Phi_{\left(  \lambda_{_{0}%
},...,\lambda_{_{3}}\right)  }(t-t_{_{j+1}})}{\Phi_{\left(  \lambda_{_{2}%
},\lambda_{_{3}}\right)  }(t_{_{j}}-t_{_{j+1}})}-A_{_{j}}^{^{\left[  2\right]
}}(t)\frac{\Phi_{\left(  \lambda_{_{0}},...,\lambda_{_{3}}\right)  }(t_{_{j}%
}-t_{_{j+1}})}{\Phi_{\left(  \lambda_{_{2}},\lambda_{_{3}}\right)  }(t_{_{j}%
}-t_{_{j+1}})}%
\]
is the desired solution. Similar arguments show that%
\[
B_{_{j}}^{^{\left[  1\right]  }}(t)=\frac{\Phi_{\left(  \lambda_{_{0}%
},...,\lambda_{_{3}}\right)  }(t-t_{_{j}})}{\Phi_{\left(  \lambda_{_{2}%
},\lambda_{_{3}}\right)  }(t_{_{j+1}}-t_{_{j}})}-B_{_{j}}^{^{\left[  2\right]
}}(t)\frac{\Phi_{\left(  \lambda_{_{0}},...,\lambda_{_{3}}\right)  }%
(t_{_{j+1}}-t_{_{j}})}{\Phi_{\left(  \lambda_{_{2}},\lambda_{_{3}}\right)
}(t_{_{j+1}}-t_{j})}.
\]
To see why uniqueness holds we refer to two obvious observations. We first
note that $L_{_{i}}v\equiv0$ in $[t_{_{j}},t_{_{j+1}}]$ together with
$v\left(  t_{_{j}}\right)  =v\left(  t_{_{j+1}}\right)  =0$ for $i=1,2$ imply
$v(t)\equiv0$ in $\left[  t_{_{j}},t_{_{j+1}}\right]  $. Applying this claim
twice consecutively we obtain that $L_{_{2}}\left(  L_{_{1}}w\right)  \equiv0$
in $\left[  t_{_{j}},t_{_{j+1}}\right]  $ together with $w\left(  t_{_{j}%
}\right)  =w\left(  t_{_{j+1}}\right)  =L_{_{1}}(w(t_{_{j}}))=L_{_{1}}\left(
w\left(  t_{_{j+1}}\right)  \right)  =0$ imply $w\left(  t\right)  \equiv0$ in
$\left[  t_{_{j}},t_{_{j+1}}\right]  $. Then uniqueness follows by contradiction.
\end{proof}

The proof of the following simple result is omitted:

\begin{proposition}
\label{Proppsij}Let $\lambda_{_{0}},...,\lambda_{_{3}}$ be complex numbers and
$t_{_{1}}<....<t_{_{n}}$ be such that (\ref{PhiCond}) holds. Let $\Lambda
_{3}=\left(  \lambda_{_{0}},...,\lambda_{_{3}}\right)  .$ Then any $L$-spline
$g(\cdot)$ with $L=L_{3}$ restricted to the interval $\left(  t_{_{j}%
},t_{_{j+1}}\right)  $ is equal to
\begin{equation}
\psi_{_{j}}(t):=L_{_{1}}g(t_{_{j}})A_{_{j}}^{^{\left[  1\right]  }}%
(t)+L_{_{1}}g(t_{_{j+1}})B_{_{j}}^{^{\left[  1\right]  }}(t)+g(t_{_{j}%
})A_{_{j}}^{^{\left[  2\right]  }}(t)+g(t_{_{j+1}})B_{_{j}}^{^{\left[
2\right]  }}(t). \label{eqpsij}%
\end{equation}

Using the notations in formula (\ref{gamma}) we may rewrite (\ref{eqpsij}) as
\begin{equation}
\psi_{_{j}}(t):=\gamma_{_{j}}A_{_{j}}^{^{\left[  1\right]  }}(t)+\gamma
_{_{j+1}}B_{_{j}}^{^{\left[  1\right]  }}(t)+g_{_{j}}A_{_{j}}^{^{\left[
2\right]  }}(t)+g_{_{j+1}}B_{_{j}}^{^{\left[  2\right]  }}(t). \label{eqpsij2}%
\end{equation}

\end{proposition}

The following is the main result of this section:

\begin{theorem}
\label{Thm1}Let $\lambda_{_{0}},...,\lambda_{_{3}}$ be complex numbers and
$t_{_{1}}<....<t_{_{n}}$ be such that (\ref{PhiCond}) holds. Let $g$ be a
natural spline for $\Lambda_{_{3}}=(\lambda_{_{0}},\lambda_{_{1}}%
,\lambda_{_{2}},\lambda_{_{3}})$ with data points $g_{_{1}},...,g_{_{n}}$ and
Then there exists an $n\times(n-2)$ matrix $Q$ and a $(n-2)\times(n-2)$ matrix
$R$ such that
\begin{equation}
Q^{T}\mathbf{g}=R\gamma, \label{QTgRgamma}%
\end{equation}
where $\gamma$ has been defined in (\ref{gamma}).
\end{theorem}

\begin{proof}
Let $g$ be a natural spline for $L_{_{2}}L_{_{1}}$ with data points $g_{_{1}%
},...,g_{_{n}}$ (we remind that $L_{_{1}}g(t_{_{1}})=L_{_{1}}g(t_{_{n}})=0$).
Then for $t\in(t_{_{j}},t_{_{j+1}})$ and $j=1,...,n-1$ the functional value
$g(t)$ is equal to
\[
\psi_{_{j}}(t)=L_{_{1}}g(t_{_{j}})A_{_{j}}^{^{\left[  1\right]  }}(t)+L_{_{1}%
}g(t_{_{j+1}})B_{j}^{^{\left[  1\right]  }}(t)+g(t_{_{j}})A_{_{j}}^{^{\left[
2\right]  }}(t)+g(t_{_{j+1}})B_{_{j}}^{^{\left[  2\right]  }}(t).
\]
Further on, $g(t)$ is equal to $\psi_{_{j-1}}(t)$ on the interval $(t_{_{j-1}%
},t_{_{j}})$. Since $g$ has a continuous derivative at $t_{_{j}}$ we conclude
that for every $j=2,....n-1$
\[
\frac{d}{dt}\psi_{_{j-1}}(t_{_{j}})=L_{_{1}}g(t_{_{j-1}})\frac{dA_{_{j-1}%
}^{^{\left[  1\right]  }}}{dt}(t_{_{j}})+L_{_{1}}g(t_{_{j}})\frac{dB_{_{j-1}%
}^{^{\left[  1\right]  }}}{dt}(t_{_{j}})+g(t_{_{j-1}})\frac{dA_{_{j-1}%
}^{^{\left[  2\right]  }}}{dt}(t_{_{j}})+g(t_{_{j}})\frac{dB_{_{j-1}%
}^{^{\left[  2\right]  }}}{dt}(t_{_{j}})
\]
is equal to
\[
\frac{d}{dt}\psi_{_{j}}(t_{_{j}})=L_{_{1}}g(t_{_{j}})\frac{dA_{_{j}%
}^{^{\left[  1\right]  }}}{dt}(t_{_{j}})+L_{_{1}}g(t_{_{j+1}})\frac{dB_{_{j}%
}^{^{\left[  1\right]  }}}{dt}(t_{_{j}})+g(t_{_{j}})\frac{dA_{_{j}}^{^{\left[
2\right]  }}}{dt}(t_{_{j}})+g(t_{_{j+1}})\frac{dB_{_{j}}^{^{\left[  2\right]
}}}{dt}(t_{_{j}}).
\]
We bring the summands containing $L_{_{1}}g^{\prime}s$ to the left hand side
and the summands containing $g^{\prime}s$ to the right hand side: then
\[
\widetilde{R}_{_{j}}:=L_{_{1}}g(t_{_{j-1}})\frac{dA_{_{j-1}}^{^{\left[
1\right]  }}}{dt}(t_{_{j}})+L_{_{1}}g(t_{_{j}})\frac{dB_{_{j-1}}^{^{\left[
1\right]  }}}{dt}(t_{_{j}})-L_{_{1}}g(t_{_{j}})\frac{dA_{_{j}}^{^{\left[
1\right]  }}}{dt}(t_{_{j}})-L_{_{1}}g(t_{_{j+1}})\frac{dB_{_{j}}^{^{\left[
1\right]  }}}{dt}(t_{_{j}})
\]
is equal to
\[
\widetilde{Q}_{_{j}}:=g(t_{_{j}})\frac{dA_{_{j}}^{^{\left[  2\right]  }}}%
{dt}(t_{_{j}})+g(t_{_{j+1}})\frac{dB_{_{j}}^{^{\left[  2\right]  }}}%
{dt}(t_{_{j}})-g(t_{_{j-1}})\frac{dA_{_{j-1}}^{^{\left[  2\right]  }}}%
{dt}(t_{_{j}})-g(t_{_{j}})\frac{dB_{_{j-1}}^{^{\left[  2\right]  }}}%
{dt}(t_{_{j}}).
\]
The number $\widetilde{R}_{_{j}}$ is an inner product of the $j$-th row of a
$\left(  n-2\right)  \times\left(  n-2\right)  $ matrix $R$ and the vector
$\left(  L_{_{1}}g(t_{_{j}})\right)  _{j=2,...n-1}^{T}$ (recall that
$L_{1}g\left(  t_{1}\right)  =L_{1}g\left(  t_{n}\right)  =0$ since $g$ is a
natural spline) where the diagonal entries $R_{_{j,j}}$ of $R$ are defined by
\[
R_{_{j,j}}=\frac{dB_{_{j-1}}^{^{\left[  1\right]  }}}{dt}(t_{_{j}}%
)-\frac{dA_{_{j}}^{^{\left[  1\right]  }}}{dt}(t_{_{j}})
\]
and the nonzero off-diagonal entries
\[
R_{_{j,j+1}}=-\frac{dB_{_{j}}^{^{\left[  1\right]  }}}{dt}(t_{_{j}})\text{ and
}R_{_{j,j-1}}=\frac{dA_{_{j-1}}^{^{\left[  1\right]  }}}{dt}(t_{_{j}}).
\]
Similarly we see that $\widetilde{Q}_{_{j}}$ is the inner product of the row
vector
\[
\left(  0,0,...,0,-\frac{dA_{_{j-1}}^{^{\left[  2\right]  }}}{dt}(t_{_{j}%
}),\quad\frac{dA_{_{j}}^{^{\left[  2\right]  }}}{dt}(t_{_{j}})-\frac
{dB_{_{j-1}}^{^{\left[  2\right]  }}}{dt}(t_{_{j}}),\quad\frac{dB_{_{j}%
}^{^{\left[  2\right]  }}}{dt}(t_{_{j}}),0,..,0\right)
\]
with the column vector $\left(  g\left(  t_{_{j}}\right)  \right)
_{j=1,...,n}^{T}$
\end{proof}

\begin{theorem}
\label{r_jj} Let $\lambda_{_{0}},...,\lambda_{_{3}}$ be complex numbers and
$t_{_{1}}<...<t_{_{n}}$ be such that (\ref{PhiCond}) holds. Then the
tridiagonal matrix $R$ is given by%
\begin{align*}
R_{_{j,j}}  &  =\rho(t_{_{j}}-t_{_{j-1}})-\rho(-(t_{_{j+1}}-t_{_{j}}))\\
R_{_{j,j+1}}  &  =\sigma(t_{_{j+1}}-t_{_{j}})\text{ and }R_{_{j+1,j}}%
=-\sigma(-(t_{_{j+1}}-t_{_{j}}))
\end{align*}
where $\rho$ and $\sigma$ are defined in (\ref{eqDefw}) and (\ref{eqDefr}).
\end{theorem}

\begin{proof}
According to (\ref{eqdefAB}) we see that
\begin{equation}
\frac{dA_{_{j}}^{^{\left[  2 \right]  }}}{dt} ( t ) = \frac{\Phi_{(
\lambda_{_{0}},\lambda_{_{1}} )}^{\prime} ( t-t_{_{j+1}} )} {\Phi_{(
\lambda_{_{0}},\lambda_{_{1}})} ( t_{_{j}} - t_{_{j+1}} )} \text{ and }
\frac{dB_{_{j}}^{^{\left[  2 \right]  }}}{dt} ( t ) = \frac{\Phi_{(
\lambda_{_{0}},\lambda_{_{1}} )}^{\prime} (t-t_{_{j}})} {\Phi_{( \lambda
_{_{0}},\lambda_{_{1}} )} ( t_{_{j+1}}-t_{_{j}} )}. \label{eqDerAB}%
\end{equation}
Since
\[
\frac{dB_{_{j}}^{^{\left[  1 \right]  }}}{dt} ( t ) = \frac{\Phi_{(
\lambda_{_{0}},...,\lambda_{_{3}} )}^{\prime} ( t-t_{_{j}} )} {\Phi_{(
\lambda_{_{2}},\lambda_{_{3}} )} ( t_{_{j+1}}-t_{_{j}} )} - \frac{\Phi_{(
\lambda_{_{0}},...,\lambda_{_{3}})} ( t_{_{j+1}}-t_{_{j}} )} {\Phi_{(
\lambda_{_{2}},\lambda_{_{3}} )} ( t_{_{j+1}}-t_{_{j}} )} \frac{\Phi_{(
\lambda_{_{0}},\lambda_{_{1}} )}^{\prime} ( t-t_{_{j}} )} {\Phi_{(
\lambda_{_{0}},\lambda_{_{1}})} ( t_{_{j+1}}-t_{_{j}})}
\]
we obtain
\[
R_{_{j,j+1}} = -\frac{dB_{_{j}}^{^{\left[  1 \right]  }}}{dt} ( t_{_{j}} ) =
\frac{\Phi_{( \lambda_{_{0}},...,\lambda_{_{3}} )} ( t_{_{j+1}}-t_{_{j}} )}
{\Phi_{( \lambda_{_{2}},\lambda_{_{3}} )} ( t_{_{j+1}}-t_{_{j}} )} \frac
{1}{\Phi_{( \lambda_{_{0}},\lambda_{_{1}})} ( t_{_{j+1}}-t_{_{j}} )} = \sigma(
t_{_{j+1}}-t_{_{j}} ).
\]
Similarly it follows that
\[
R_{_{j+1,j}} = \frac{dA_{_{j}}^{^{\left[  1 \right]  }}}{dt} ( t_{_{j+1}} ) =
-\frac{\Phi_{( \lambda_{_{0}},...,\lambda_{_{3}} )} ( t_{_{j}}-t_{_{j+1}} )}
{\Phi_{( \lambda_{_{2}},\lambda_{_{3}} )} ( t_{_{j}}-t_{_{j+1}} )} \frac
{1}{\Phi_{( \lambda_{_{0}}, \lambda_{_{1}} )} ( t_{_{j}}-t_{_{j+1}} )} =
-\sigma( -( t_{_{j+1}}-t_{_{j}}) ).
\]
Next we consider
\[
R_{_{j,j}} = \frac{dB_{_{j-1}}^{^{\left[  1 \right]  }}}{dt} ( t_{_{j}} ) -
\frac{dA_{_{j}}^{^{\left[  1 \right]  }}}{dt} ( t_{_{j}} ).
\]
The formula for $R_{_{jj}}$ in the theorem now follows from
\[
\frac{dB_{_{j-1}}^{^{\left[  1 \right]  }}}{dt} ( t_{_{j}} ) = \frac{\Phi_{(
\lambda_{_{0}},...,\lambda_{_{3}} )}^{\prime}( t_{_{j}}-t_{_{j-1}} )} {\Phi_{(
\lambda_{_{2}},\lambda_{_{3}} )} ( t_{_{j}}-t_{_{j-1}} )} - \frac{\Phi_{(
\lambda_{_{0}},...,\lambda_{_{3}} )} ( t_{_{j}}-t_{_{j-1}} )} {\Phi_{(
\lambda_{_{2}},\lambda_{_{3}} )} ( t_{_{j}}-t_{_{j-1}} )} \frac{\Phi_{(
\lambda_{_{0}},\lambda_{_{1}} )}^{\prime}( t_{_{j}}-t_{_{j-1}} )} {\Phi_{(
\lambda_{_{0}},\lambda_{_{1}} )} ( t_{_{j}}-t_{_{j-1}} )}
\]
and
\[
\frac{dA_{_{j}}^{^{\left[  1 \right]  }}}{dt} ( t_{_{j}} ) = \frac{\Phi_{(
\lambda_{_{0}},...,\lambda_{_{3}} )}^{\prime}( t_{_{j}}-t_{_{j+1}} )} {\Phi_{(
\lambda_{_{2}},\lambda_{_{3}} )} ( t_{_{j}}-t_{_{j+1}} )} - \frac{\Phi_{(
\lambda_{_{0}},...,\lambda_{_{3}} )} ( t_{_{j}}-t_{_{j+1}} )} {\Phi_{(
\lambda_{_{2}},\lambda_{_{3}} )} ( t_{_{j}}-t_{_{j+1}} )} \frac{\Phi_{(
\lambda_{_{0}},\lambda_{_{1}} )}^{\prime}( t_{_{j}}-t_{_{j+1}} )} {\Phi_{(
\lambda_{_{0}},\lambda_{_{1}} )} ( t_{_{j}}-t_{_{j+1}} )}.
\]

\end{proof}

In the same spirit one proves the following result:

\begin{theorem}
\label{ThmQ} The $n\times(n-2)$ matrix $Q$ is given by
\begin{equation}
q_{_{j-1,j}}=-\frac{1}{\Phi_{(\lambda_{_{0}},\lambda_{_{1}})}(t_{_{j-1}%
}-t_{_{j}})},\qquad q_{_{j+1,j}}=\frac{1}{\Phi_{(\lambda_{_{0}},\lambda_{_{1}%
})}(t_{_{j+1}}-t_{_{j}})} \label{Qij}%
\end{equation}
and
\[
q_{_{jj}}=\frac{\Phi_{(\lambda_{_{0}},\lambda_{_{1}})}^{\prime}(t_{_{j}%
}-t_{_{j+1}})}{\Phi_{(\lambda_{_{0}},\lambda_{_{1}})}(t_{_{j}}-t_{_{j+1}}%
)}-\frac{\Phi_{(\lambda_{_{0}},\lambda_{_{1}})}^{\prime}(t_{_{j}}-t_{_{j-1}}%
)}{\Phi_{(\lambda_{_{0}},\lambda_{_{1}})}(t_{_{j}}-t_{_{j-1}})}%
\]
for $j=2,3,...,n-1$.
\end{theorem}

\begin{proof}
According to (\ref{eqDerAB}) we first see that
\[
\frac{dA_{_{j}}^{^{\left[  2 \right]  }}}{dt} ( t ) = \frac{\Phi_{(
\lambda_{_{0}},\lambda_{_{1}} )}^{\prime}( t-t_{_{j+1}} )} {\Phi_{(
\lambda_{_{0}},\lambda_{_{1}} )} ( t_{_{j}}-t_{_{j+1}} )} \text{ and }
\frac{dB_{_{j}}^{^{\left[  2 \right]  }}}{dt} ( t ) = \frac{\Phi_{(
\lambda_{_{0}},\lambda_{_{1}} )}^{\prime} ( t-t_{_{j}} )} { \Phi_{(
\lambda_{_{0}},\lambda_{_{1}} )} ( t_{_{j+1}}-t_{_{j}} )}.
\]
It follows that
\begin{align*}
q_{_{j-1,j}}  &  = - \frac{dA_{_{j-1}}^{^{\left[  2 \right]  }}}{dt} (
t_{_{j}} ) = - \frac{\Phi_{( \lambda_{_{0}},\lambda_{_{1}} )}^{\prime} (
t_{_{j}}-t_{_{j}} )} {\Phi_{( \lambda_{_{0}},\lambda_{_{1}} )} ( t_{_{j-1}%
}-t_{_{j}} )} = - \frac{1}{\Phi_{( \lambda_{_{0}},\lambda_{_{1}} )} (
t_{_{j-1}}-t_{_{j}} )}\\
q_{_{j+1,j}}  &  = \frac{dB_{_{j}}^{^{\left[  2 \right]  }}}{dt} ( t_{_{j}} )
= \frac{ \Phi_{( \lambda_{_{0}},\lambda_{_{1}} )}^{\prime} ( t_{_{j}}-t_{_{j}}
)} {\Phi_{( \lambda_{_{0}},\lambda_{_{1}} )} ( t_{_{j+1}}-t_{_{j}} )} =
\frac{1}{ \Phi_{( \lambda_{_{0}},\lambda_{_{1}} )} ( t_{_{j+1}}-t_{_{j}} )}%
\end{align*}
and
\[
q_{_{jj}} = \frac{dA_{_{j}}^{^{\left[  2 \right]  }}}{dt} ( t_{_{j}} ) -
\frac{dB_{_{j-1}}^{^{\left[  2 \right]  }}}{dt} ( t_{_{j}} ) = \frac{\Phi_{(
\lambda_{_{0}},\lambda_{_{1}} )}^{\prime} ( t_{_{j}}-t_{_{j+1}} )} {\Phi_{(
\lambda_{_{0}},\lambda_{_{1}} )} ( t_{_{j}}-t_{_{j+1}} )} - \frac{\Phi_{(
\lambda_{_{0}},\lambda_{_{1}} )}^{\prime} ( t_{_{j}}-t_{_{j-1}} )} {\Phi_{(
\lambda_{_{0}},\lambda_{_{1}} )} ( t_{_{j}}-t_{_{j-1}} )}
\]

\end{proof}

So far we have derived formulae connecting the values $\mathbf{g}=\left(
g\left(  t_{_{1}}\right)  ,...,g\left(  t_{_{n}}\right)  \right)  ^{T}$ and
the derivatives $\gamma=\left(  L_{_{1}}g\left(  t_{_{2}}\right)
,...,L_{_{1}}g\left(  t_{_{n-1}}\right)  \right)  ^{T}$ of the $L-$spline $g$.
The arguments can be reversed to construct the corresponding $L$-spline $g$
interpolating given data $z=\left(  z_{_{1}},...,z_{_{n}}\right)  ^{T}$ as we
shall now outline:

\begin{enumerate}
\item[Step 1] Set $g\left(  t_{i}\right)  =z_{_{i}}$ for $i=1,..n,$

\item[Step 2] Set $x=Q^{T}\mathbf{g}$ and solve $R\gamma=x$ (assuming that the
matrix $R$ is invertible).
\end{enumerate}

Then we use formula (\ref{eqpsij2}) to recover the values of the $L-$spline on
every interval $\left[  t_{j},t_{j+1}\right]  .$

Thus the solvability of the equation $R\gamma=x$ is a key property for the
existence of interpolation $L-$splines for arbitrary interpolation data $g.$
The fact that the matrix $R$ is tridiagonal and $Q$ is tridiagonal-like imply
that the number of operations for the computation of the $L-$spline $g$ is
$O\left(  n\right)  .$

The Levy--Desplanques theorem says that a strictly diagonally dominant matrix
(or an irreducibly diagonally dominant matrix) is invertible, see
\cite{HornJohnson}, Theorem 6.1.10. In the following section 4 and section 5 we shall
provide sufficient criteria such that the matrix $R$ is strictly diagonally
dominant, and in particular invertible.

\section{Diagonal dominance of $R$ in the general case}

At first we present a simple criterion for diagonal dominance of the matrix
$R$ .

\begin{theorem}
\label{bound} \label{ThmDom}Let $\left(  \lambda_{0},\lambda_{1},\lambda
_{2},\lambda_{3}\right)  \in\mathbb{C}^{4}$ and assume that $\rho(x)$ and
$-\rho(-x)$ are positive on the interval $(0,\delta)$, and let $M_{\delta}>0$
be such that
\begin{equation}
\left\vert \sigma(-x)\right\vert \leq M_{\delta}\left\vert \rho(x)\right\vert
\text{ for all }x\in\left[  -\delta,\delta\right]  . \label{eqsigest}%
\end{equation}
Then the estimate
\[
\left\vert R_{_{j,j-1}}\right\vert +\left\vert R_{_{j,j+1}}\right\vert \leq
M_{\delta}\left\vert R_{_{j,j}}\right\vert \text{ }%
\]
holds for all partitions $t_{_{1}}<....<t_{_{n}}$ such that $t_{_{j+1}%
}-t_{_{j}}\leq\delta$ for $j=1,...,n-1.$
\end{theorem}

\begin{proof}
Since $R_{_{j,j+1}}=\sigma(t_{_{j+1}}-t_{_{j}})$ and $R_{_{j,j-1}}%
=-\sigma(-(t_{_{j}}-t_{_{j-1}}))$ we have the general formula:
\[
\frac{\left\vert R_{_{j,j-1}}\right\vert +\left\vert R_{_{j,j+1}}\right\vert
}{\left\vert R_{_{j,j}}\right\vert }=\frac{\left\vert \sigma(t_{_{j+1}%
}-t_{_{j}})\right\vert +\left\vert \sigma(-(t_{_{j}}-t_{_{j-1}}))\right\vert
}{\left\vert \rho(t_{_{j}}-t_{_{j-1}})-\rho(-(t_{_{j+1}}-t_{_{j}}))\right\vert
}%
\]
Using the estimate (\ref{eqsigest}) we obtain
\begin{equation}
\frac{\left\vert R_{_{j,j-1}}\right\vert +\left\vert R_{_{j,j+1}}\right\vert
}{\left\vert R_{_{j,j}}\right\vert }\leq M_{\delta}\frac{\left\vert
\rho(-(t_{_{j+1}}-t_{_{j}}))\right\vert +\left\vert \rho(t_{_{j}}-t_{_{j-1}%
})\right\vert }{\left\vert \rho(t_{_{j}}-t_{_{j-1}})-\rho(-(t_{_{j+1}}%
-t_{_{j}}))\right\vert }. \label{eqMdelta}%
\end{equation}
for all partitions $t_{_{1}}<....<t_{_{n}}$ such that $t_{_{j+1}}-t_{_{j}}%
\leq\delta$ for $j=1,...,n.$ Since we assume that $\rho(x)$ and $-\rho(-x)$
are positive on the interval $(0,\delta)$ the fraction on the right hand side
in (\ref{eqMdelta}) is equal to $1.$
\end{proof}

The sufficient criterion reduces the problem of finding the best constant
$M_{\delta}$ for strict diagonal dominance to the determination of the
absolute maximum of the function
\[
\tau\left(  x\right)  :=\frac{-\sigma\left(  -x\right)  }{\rho\left(
x\right)  }\text{ for }x\in\mathbb{R}.
\]
It is convenient to consider the numerators of the functions $\rho(x)$ and
$\sigma(x)$ given by
\begin{align}
\rho_{_{0}}(x)  &  =\Phi_{(\lambda_{_{0}},...,\lambda_{_{3}})}^{\prime}%
(x)\Phi_{(\lambda_{_{0}},\lambda_{_{1}})}(x)-\Phi_{(\lambda_{_{0}}%
,...,\lambda_{_{3}})}(x)\Phi_{(\lambda_{_{0}},\lambda_{_{1}})}^{\prime
}(x)\label{eqDefrhonull}\\
\sigma_{_{0}}(x)  &  =\Phi_{(\lambda_{_{0}},...,\lambda_{_{3}})}(x).
\label{eqDefsigmanull}%
\end{align}
The Taylor expansion of $\Phi_{\left(  \lambda_{0},...,\lambda_{N}\right)  }$
has been discussed in \cite{KoRe13}:
\begin{equation}
\Phi_{\left(  \lambda_{0},...,\lambda_{3}\right)  }\left(  x\right)
=\frac{x^{3}}{3!}+\frac{A}{4!}x^{4}+\frac{B}{5!}x^{5}+g_{6}\left(  x\right)
\label{eqTaylor1}%
\end{equation}
where $g_{6}$ has a zero of order $6$ and $A=\lambda_{0}+\lambda_{1}%
+\lambda_{2}+\lambda_{3}$ and
\begin{equation}
B=\lambda_{0}^{2}+\lambda_{1}^{2}+\lambda_{2}^{2}+\lambda_{3}^{2}+\lambda
_{0}\lambda_{1}+\lambda_{0}\lambda_{2}+\lambda_{0}\lambda_{3}+\lambda
_{1}\lambda_{2}+\lambda_{1}\lambda_{3}+\lambda_{2}\lambda_{3}\allowbreak.
\label{eqDefB}%
\end{equation}
It is easy to check that
\begin{equation}
B=\frac{1}{2}A^{2}+\frac{1}{2}\left(  \lambda_{0}^{2}+\lambda_{1}^{2}%
+\lambda_{2}^{2}+\lambda_{3}^{2}\right)  \label{eqDefB2}%
\end{equation}

\begin{proposition}
\label{PropositionMdeltaHalf}Let $\left(  \lambda_{_{0}},\lambda_{_{1}%
},\lambda_{_{2}},\lambda_{_{3}}\right)  \in\mathbb{C}^{4}$ and define
$A=\lambda_{0}+\lambda_{1}+\lambda_{2}+\lambda_{3}$ and $C=$ $\lambda
_{0}+\lambda_{1}.$ Then
\begin{equation}
\tau\left(  x\right)  =\frac{-\sigma(-x)}{\rho(x)}=\frac{\Phi_{\left(
A-\lambda_{0},....,A-\lambda_{3}\right)  }\left(  x\right)  }{\rho_{0}(x)}.
\label{eqDeftau}%
\end{equation}
The Taylor polynomial of degree $2$ of $\tau\left(  x\right)  $ is given by
\[
\frac{1}{2}+\frac{1}{8}\left(  \frac{3A}{2}-C\right)  \allowbreak x+\frac
{1}{8}\allowbreak\left(  \frac{7A^{2}}{16}-\frac{AC}{2}+\frac{C^{2}}{4}%
-\frac{B}{5}\right)  \allowbreak x^{2}.
\]

\end{proposition}

\begin{proof}
By Proposition \ref{PropSym} we see that%
\[
\tau\left(  x\right)  =\frac{-\sigma\left(  -x\right)  }{\rho\left(  x\right)
}=\frac{-\sigma_{0}\left(  -x\right)  \Phi_{(\lambda_{_{0}},\lambda_{_{1}}%
)}(x)\Phi_{(\lambda_{_{2}},\lambda_{_{3}})}(x)}{\rho_{0}\left(  x\right)
\Phi_{(\lambda_{_{0}},\lambda_{_{1}})}(-x)\Phi_{(\lambda_{_{2}},\lambda_{_{3}%
})}(-x)}=\frac{-\sigma_{0}\left(  -x\right)  e^{\left(  \lambda_{0}%
+\cdots+\lambda_{3}\right)  x}}{\rho_{0}\left(  x\right)  }%
\]
and (\ref{eqDeftau}) follows from the identity
\[
-\sigma_{0}\left(  -x\right)  e^{\left(  \lambda_{0}+\cdots+\lambda
_{3}\right)  x}=\Phi_{\left(  -\lambda_{0},....,-\lambda_{3}\right)  }\left(
x\right)  e^{Ax}=\Phi_{\left(  A-\lambda_{0},....,A-\lambda_{3}\right)
}\left(  x\right)  .
\]
The Taylor expansion of $\Phi_{(\lambda_{_{0}},\lambda_{_{1}})}$ is of the
form
\[
\Phi_{(\lambda_{_{0}},\lambda_{_{1}})}(x)=x+\frac{C}{2!}x^{2}+\frac{D}%
{3!}x^{3}+h_{4}\left(  x\right)
\]
where $C=\lambda_{0}+\lambda_{1}$ and $D=\lambda_{0}^{2}+\lambda_{0}%
\lambda_{1}+\lambda_{1}^{2}$ and $h_{4}\left(  x\right)  $ has a zero of order
$4$ at $x=0.$ This can be directly seen for $\lambda_{1}\neq\lambda_{0}$ from
the Taylor expansion of (\ref{eqdegPhitwo}). Then with $g_{5}=g_{6}^{\prime
}\ $and $h_{3}:=h_{4}^{\prime}$ we can write
\begin{align*}
\rho_{0}\left(  x\right)   &  =\left(  \frac{x^{2}}{2!}+\frac{A}{3!}%
x^{3}+\frac{B}{4!}x^{4}+g_{5}\right)  \left(  x+\frac{C}{2}x^{2}+\frac{D}%
{3!}x^{3}+h_{4}\right) \\
&  -\left(  \frac{x^{3}}{3!}+\frac{A}{4!}x^{4}+\frac{B}{5!}x^{5}+g_{6}\right)
\left(  1+Cx+\frac{D}{2}x^{2}+h_{3}\right)  .
\end{align*}
It follows that
\[
\rho_{0}\left(  x\right)  =\frac{1}{3}x^{3}+\left(  \frac{1}{8}A+\frac{1}%
{12}C\right)  x^{4}+\left(  \frac{1}{30}B+\frac{1}{24}AC\right)  \allowbreak
x^{5}+G_{6}\left(  x\right)
\]
for some $G_{6}$ which has a zero of order $6.$ Now we apply formula
(\ref{eqTaylor1}) applied to the vector $\left(  \widetilde{\lambda}%
_{0},...,\widetilde{\lambda}_{3}\right)  =\left(  A-\lambda_{0},....,A-\lambda
_{3}\right)  ,$ providing the expansion
\[
\Phi_{\left(  A-\lambda_{0},....,A-\lambda_{3}\right)  }\left(  x\right)
=\frac{1}{3!}x^{3}+\frac{\widetilde{A}}{4!}x^{4}+\frac{\widetilde{B}}{5!}%
x^{5}+r_{6}\left(  x\right)
\]
where $\widetilde{A}=\widetilde{\lambda}_{0}+\cdots+\widetilde{\lambda}%
_{3}=3A$ and using (\ref{eqDefB2})
\begin{align*}
\widetilde{B}  &  =\frac{1}{2}\widetilde{A}^{2}+\frac{1}{2}\left(
\widetilde{\lambda}_{0}^{2}+\widetilde{\lambda}_{1}^{2}+\widetilde{\lambda
}_{2}^{2}+\widetilde{\lambda}_{3}^{2}\right) \\
&  =\frac{9}{2}A^{2}+\frac{1}{2}\left(  (A^{2}-2\lambda_{0}A+\lambda_{0}%
^{2})+\ldots+(A^{2}-2\lambda_{3}A+\lambda_{3}^{2})\right) \\
&  =\frac{9}{2}A^{2}+2A^{2}-A^{2}+\frac{1}{2}\left(  \lambda_{0}^{2}%
+\lambda_{1}^{2}+\lambda_{2}^{2}+\lambda_{3}^{2}\right)  =5A^{2}+B
\end{align*}
Hence we have%
\[
\frac{\Phi_{\left(  A-\lambda_{0},....,A-\lambda_{3}\right)  }\left(
x\right)  }{\rho_{0}\left(  x\right)  }=\frac{\frac{1}{3!}x^{3}+\frac{3A}%
{4!}x^{4}+\frac{5A^{2}+B}{5!}x^{5}+r_{6}\left(  x\right)  }{\frac{1}{3}%
x^{3}+\left(  \frac{1}{8}A+\frac{1}{12}C\right)  x^{4}+\left(  \frac{1}%
{30}B+\frac{1}{24}AC\right)  \allowbreak x^{5}+G_{6}}.
\]
Cancel the power $x^{3}$ and write
\[
\frac{\Phi_{\left(  A-\lambda_{0},....,A-\lambda_{3}\right)  }\left(
x\right)  }{\rho_{0}\left(  x\right)  }=\frac{p_{2}+G_{3}}{q_{2}+R_{3}}%
\]
where $p_{2}=\frac{1}{3!}+\frac{3A}{4!}x+\frac{5A^{2}+B}{5!}x^{2}$ and
$q_{2}=\frac{1}{3}+\left(  \frac{1}{8}A+\frac{1}{12}C\right)  x+\left(
\frac{1}{30}B+\frac{1}{24}AC\right)  \allowbreak x^{2}.$ Note that
\[
\frac{p_{2}+G_{3}}{q_{2}+R_{3}}-\frac{p_{2}}{q_{2}}=\frac{\left(  p_{2}%
+G_{3}\right)  q_{2}-p_{2}\left(  q_{2}+R_{3}\right)  }{q_{2}\left(
q_{2}+R_{3}\right)  }=\frac{G_{3}q_{2}-p_{2}R_{3}}{q_{2}\left(  q_{2}%
+R_{3}\right)  }=\frac{x^{3}H\left(  x\right)  }{q_{2}\left(  q_{2}%
+R_{3}\right)  }.
\]
Thus the Taylor polynomials of second order of $\frac{p_{2}+G_{3}}{q_{2}%
+R_{3}}$ and $\frac{p_{2}}{q_{2}}$ are identical. A computation shows that
\[
\frac{p_{2}\left(  x\right)  }{q_{2}\left(  x\right)  }\ =\frac{1}{2}+\frac
{1}{8}\left(  \frac{3}{2}A-C\right)  \allowbreak x+\frac{1}{8}\allowbreak
\left(  \frac{7}{16}A^{2}-\frac{1}{2}AC+\frac{1}{4}C^{2}-\frac{1}{5}B\right)
x^{2}+\allowbreak O\left(  x^{3}\right)  .
\]

\end{proof}

With similar methods one can show that
\[
\rho\left(  x\right)  =\allowbreak\frac{1}{3}x-\frac{\lambda_{2}+\lambda
_{3}-\lambda_{0}-\lambda_{1}}{24}x^{2}+a_{2}x^{3}+O\left(  x^{4}\right)  .
\]
where
\[
a_{2}=\ -\frac{1}{720}\left(  \lambda_{0}^{2}+\lambda_{1}^{2}+\lambda_{2}%
^{2}+\lambda_{3}^{2}-14\lambda_{0}\lambda_{1}+6\lambda_{0}\lambda_{2}%
+6\lambda_{0}\lambda_{3}+6\lambda_{1}\lambda_{2}+6\lambda_{1}\lambda
_{3}-14\lambda_{2}\lambda_{3}\right)  .\allowbreak
\]
This gives immediately the following result:

\begin{corollary}
\label{CorOdd}Let $\left(  \lambda_{_{0}},\lambda_{_{1}},\lambda_{_{2}%
},\lambda_{_{3}}\right)  \in\mathbb{C}^{4}.$ If $\rho\left(  x\right)  $ is a
real-valued for $x\in\mathbb{R}$ then $\lambda_{2}+\lambda_{3}-\lambda
_{0}-\lambda_{1}$ is a real number. If $\rho$ is an odd function then
$\lambda_{0}+\lambda_{1}=\lambda_{2}+\lambda_{3}.$
\end{corollary}

$\ $Recall that a vector $\Lambda_{N}=\left(  \lambda_{0},...,\lambda
_{N}\right)  $ is conjugation invariant if there exists a permutation
$p:\left\{  0,..N\right\}  \rightarrow\left\{  0,..N\right\}  $ such that
$\overline{\lambda_{j}}=\lambda_{p\left(  j\right)  }$ for $j=0,...,N.$ In
this case the fundamental function $\Phi_{\Lambda_{N}}$ is real-valued. If
$\left(  \lambda_{0},\lambda_{1}\right)  $ and $\left(  \lambda_{2}%
,\lambda_{3}\right)  $ are conjugation-invariant then $\left(  \lambda
_{0},\lambda_{1},\lambda_{2},\lambda_{3}\right)  $ is conjugation invariant
and we infer that
\[
\Phi_{(\lambda_{0},\lambda_{1})},\Phi_{(\lambda_{2},\lambda_{3})}%
,\Phi_{(\lambda_{0},\lambda_{1},\lambda_{2},\lambda_{3})},\Phi_{(\lambda
_{0},\lambda_{1},\lambda_{2},\lambda_{3})}^{\prime}\text{ and }\Phi
_{(\lambda_{2},\lambda_{3})}^{\prime}%
\]
are real-valued functions. Our general Condition \ref{ConditionSTAR} implies
that
\[
\Phi_{(\lambda_{0},\lambda_{1})}(t)\neq0\text{ and }\Phi_{(\lambda_{2}%
,\lambda_{3})}(t)\neq0\text{ for }t\in(0,\delta)
\]
for some $\delta>0,$ hence the functions $\rho(x)$ and $\sigma(x)$ are
real-valued and well-defined and all entries of the matrices $R$ and $Q$ are
real-valued for all $t_{1}<...<t_{n}$ such that $\left\vert t_{j+1}%
-t_{j}\right\vert <\delta$ for $j=1,...,n-1.$

Proposition \ref{PropositionMdeltaHalf} guarantees that we may find a small
enough $\delta>0$ such that $M_{\delta}<1$ which provides the strict diagonal
dominance property of the matrix $R.$ In view of that it is worth studying the
equation $M_{\delta}=1,$ for estimating the critical value of $\delta$ for
which this equality is achieved.

\begin{theorem}
\label{TheoremDIAGDominant}Suppose that $(\lambda_{_{0}},\lambda_{_{1}})$ and
$(\lambda_{2},\lambda_{3})$ are conjugation invariant and $\varepsilon>0.$
Then there exists $\delta>0$ such that the estimate
\[
\left\vert R_{_{j-1,j}}\right\vert +\left\vert R_{_{j,j+1}}\right\vert
\leq\left(  \frac{1}{2}+\varepsilon\right)  \left\vert R_{_{j,j}}\right\vert
\text{ }%
\]
holds for all partitions $t_{_{1}}<....<t_{_{n}}$ such that $t_{_{j+1}%
}-t_{_{j}}\leq\delta$ for $j=1,...,n-1.$
\end{theorem}

\begin{proof}
We know that $\Phi_{(\lambda_{_{0}},\lambda_{_{1}})}(x)$ and $\Phi
_{(\lambda_{_{2}},\lambda_{_{3}})}(x)$ are positive for all $x\in(0,\delta
_{0})$ for some $\delta_{0}>0.$ Further
\[
\rho_{0}(x)=\Phi_{(\lambda_{_{0}},...,\lambda_{_{3}})}^{\prime}(x)\Phi
_{(\lambda_{_{0}},\lambda_{_{1}})}(x)-\Phi_{(\lambda_{_{0}},...,\lambda_{_{3}%
})}(x)\Phi_{(\lambda_{_{0}},\lambda_{_{1}})}^{\prime}(x)
\]
is real-valued since $\Phi_{(\lambda_{_{0}},...,\lambda_{_{3}})}(x)$ is
real-valued. Since $\rho_{_{0}}^{^{(3)}}(0)=2$ and $\rho^{^{(j)}}(0)=0$ for
$j=0,1,2$ it follows that $\rho_{_{0}}(x)$ is positive on some interval
$(0,\delta_{_{1}})$ with $0<\delta_{_{1}}\leq\delta_{_{0}}$, hence $\rho(x)$
is positive on $(0,\delta_{_{1}})$. Since $\rho(x)$ has a simple zero $x=0$ we
conclude that $\rho(-x)$ must be negative for all $x\in(0,\delta_{2})$ for
sufficiently small $\delta_{2}>0.$ Thus $\rho(x)$ and $-\rho(-x)$ are both
positive on $(0,\delta_{_{2}})$. For given $\varepsilon>0$ we can choose
$\delta>0$ such that
\[
\left\vert \frac{\sigma(x)}{\rho(-x)}\right\vert \leq\frac{1}{2}+\varepsilon
\]
since we know that $\sigma(x)/(-\rho(-x))\rightarrow\frac{1}{2}$ for
$x\rightarrow0$. Now apply Theorem \ref{ThmDom}.
\end{proof}

\begin{corollary}
Under the assumptions of Theorem \ref{TheoremDIAGDominant}, for $\varepsilon
<1/2$ the matrix $R$ is non-singular.
\end{corollary}

The following example shows that in the case of real exponents one can not
expect that the matrix $R$ is (strictly) diagonally dominant for all
partitions $t_{_{1}}<....<t_{_{n}}.$ Indeed, the function $\tau$ can be
unbounded, can have local maxima and local extrema as the following example shows:

\begin{example}
\label{ExampleNEW1}$\Lambda_{3}=\left(  3,3,-1,-1\right)  .$ Then
$\Phi_{\left(  3,3\right)  }\left(  x\right)  =xe^{3x}$ and $\Phi_{\left(
-1,-1\right)  }\left(  x\right)  =xe^{-x}$ and it is easy to check that
\[
\Phi_{\left(  3,3,-1,-1\right)  }\left(  x\right)  =\frac{1}{32}\left(
-e^{3x}+2xe^{3x}+e^{-x}+2xe^{-x}\right)  .
\]
A computation shows that
\[
\rho_{0}\left(  x\right)  =-\frac{1}{4}x^{2}e^{2x}-\frac{1}{8}xe^{2x}+\frac
{1}{32}e^{6x}-\frac{1}{32}e^{2x}\allowbreak
\]
Then
\[
\tau\left(  x\right)  =\frac{-\sigma_{0}\left(  -x\right)  e^{\left(
\lambda_{0}+\cdots+\lambda_{3}\right)  x}}{\rho_{0}\left(  x\right)  }%
=\frac{1}{32}\frac{-\left(  -e^{-3x}-2xe^{-3x}+e^{x}-2xe^{x}\right)  e^{4x}%
}{-\frac{1}{4}x^{2}e^{2x}-\frac{1}{8}xe^{2x}+\frac{1}{32}e^{6x}-\frac{1}%
{32}e^{2x}\allowbreak}.
\]
has a local minimum at $x=0$ and a local maximum at some point $x_{0}%
\in\left(  0.7,1\right)  $. Further $\tau\left(  x\right)  $ is unbounded for
$x\rightarrow-\infty$ and $\tau\left(  x\right)  \rightarrow0$ for
$x\rightarrow0.$

\begin{center}
\fbox{\includegraphics[
height=2.8442in,
width=2.8885in,
keepaspectratio
]%
{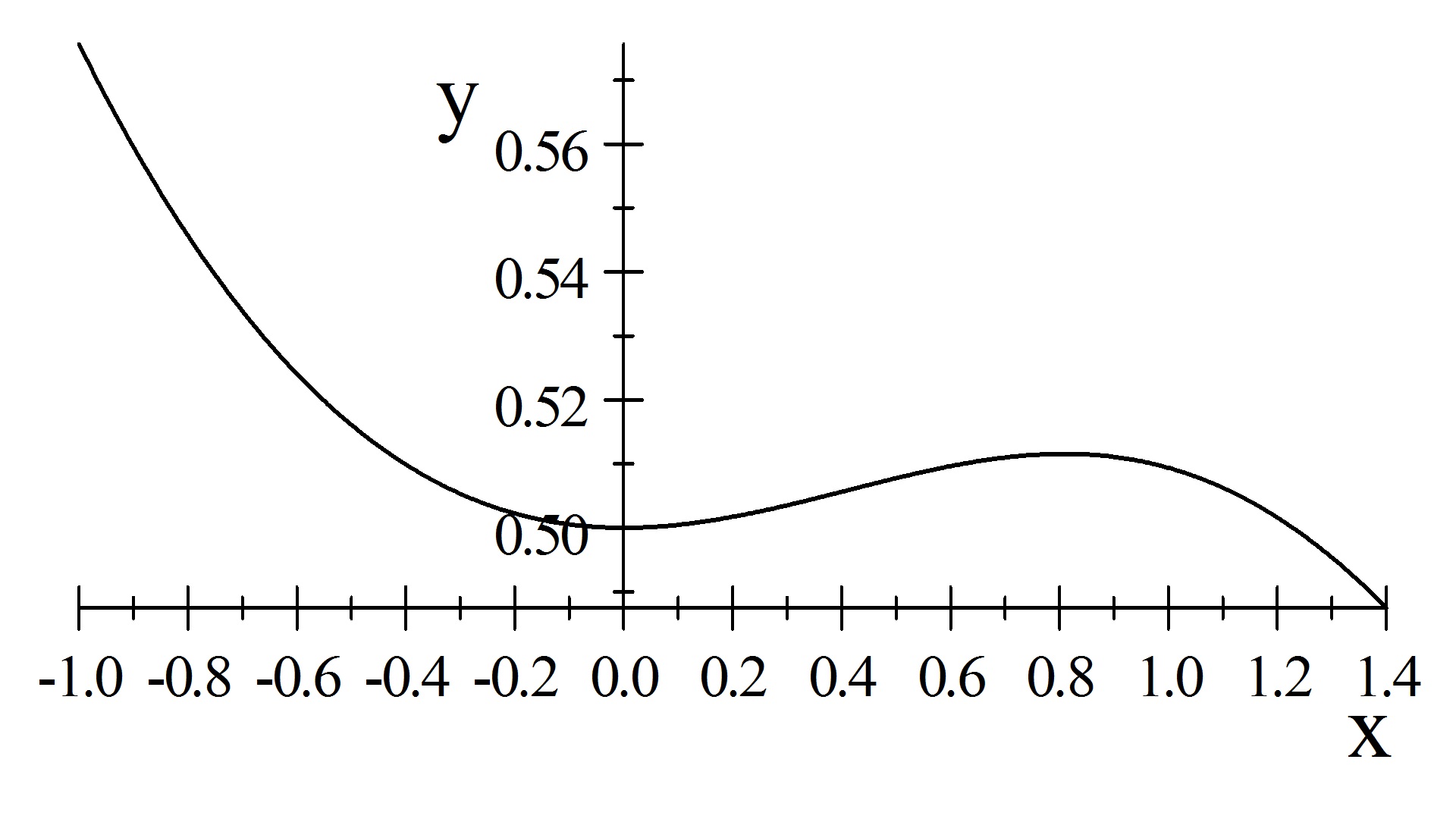}%
}\\
Fig. 1: Graph of $\tau\left(  x\right)  =-\sigma\left(  -x\right)
/\rho\left(  x\right)  $%
\end{center}

\end{example}

On the other hand, if we assume $\lambda_{0}<0<\lambda_{1}$ there is a
substantial change in the behavior of $\rho.$ In the next example we take
$\Lambda_{3}=\left(  -1,3,-1,3\right)  $ which leads to the same spline space
as in \ref{ExampleNEW1} but provides a different definition of a natural spline:

\begin{example}
Let $\Lambda_{3}=\left(  -1,3,-1,3\right)  .$ Then $\Phi_{\left(  -1,3\right)
}\left(  x\right)  =\frac{e^{3x}-e^{-x}}{4}$ and a computation shows that
$\ $
\[
\frac{-\sigma\left(  -x\right)  }{\rho\left(  x\right)  }=\frac{-\frac{1}%
{32}\left(  -e^{-3x}-2xe^{-3x}+e^{x}-2xe^{x}\right)  e^{4x}}{\frac{1}%
{64}e^{6x}-\frac{1}{64}e^{-2x}-\frac{1}{8}xe^{2x}}.
\]
In this case $R$ is strictly diagonally dominant for all partitions
$t_{1}<...<t_{n}.$

\begin{center}
\fbox{\includegraphics[
height=2.8655in,
width=2.8699in,
keepaspectratio
]%
{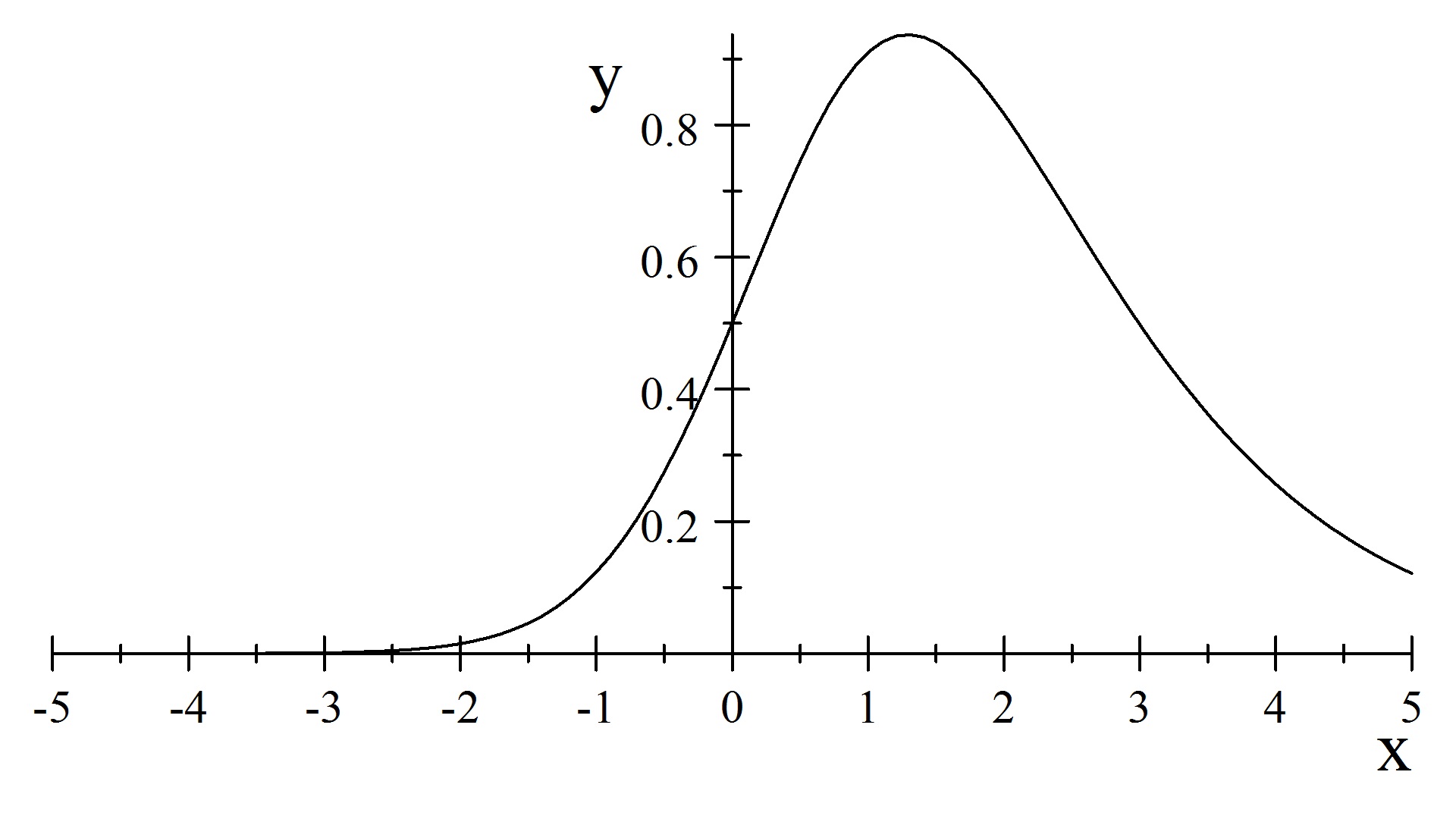}%
}\\
Fig. 2: Graph of $\tau\left(  x\right)  =$ $-\sigma(-x)/\rho(x)$%
\end{center}

\end{example}

\section{Diagonal dominance of $R$ for special exponents}

We will use the Taylor expansion of $\tau\left(  x\right)  $ in Proposition
\ref{PropositionMdeltaHalf} to show the following result:

\begin{proposition}
\label{ThmMAX}Let $\left(  \lambda_{_{0}},\lambda_{_{1}},\lambda_{_{2}%
},\lambda_{_{3}}\right)  \in\mathbb{R}^{4}.$ Then the function
\[
\tau\left(  x\right)  =\frac{-\sigma(-x)}{\rho(x)}\
\]
has a local maximum at $x=0$ provided that $\lambda_{0}+\lambda_{1}=-3\left(
\lambda_{2}+\lambda_{3}\right)  $ and
\[
\lambda_{0}^{2}-4\lambda_{0}\lambda_{1}+\lambda_{1}^{2}+3\left(  \lambda
_{2}^{2}+\lambda_{3}^{2}\right)  >0.
\]
In particular, if $\lambda_{0}\leq0\leq\lambda_{1}$ or $\lambda_{1}%
^{2}+\lambda_{0}^{2}<\lambda_{2}^{2}+\lambda_{3}^{2}$ this condition is satisfied.
\end{proposition}

\begin{proof}
From the Taylor expansion of $\tau\left(  x\right)  $ we see that $x=0\ $is
critical point if and only if
\[
0=3A-2C=3\left(  \lambda_{0}+\lambda_{1}+\lambda_{2}+\lambda_{3}\right)
-2\left(  \lambda_{0}+\lambda_{1}\right)  =\lambda_{0}+\lambda_{1}+3\left(
\lambda_{2}+\lambda_{3}\right)  .
\]
Thus we have $A=\frac{2}{3}C.$ The second Taylor coefficient is given by%
\[
a_{2}:=\frac{1}{8}\allowbreak\left(  \frac{7A^{2}}{16}-\frac{AC}{2}%
+\frac{C^{2}}{4}-\frac{B}{5}\right)  =\frac{1}{8\cdot5}\allowbreak\left(
\frac{5}{9}C^{2}-B\right)  .
\]
Further
\[
B=\frac{1}{2}A^{2}+\frac{1}{2}\left(  \lambda_{0}^{2}+\lambda_{1}^{2}%
+\lambda_{2}^{2}+\lambda_{3}^{2}\right)  =\frac{2}{9}C^{2}+\frac{1}{2}\left(
\lambda_{0}^{2}+\lambda_{1}^{2}+\lambda_{2}^{2}+\lambda_{3}^{2}\right)  .
\]
Thus $a_{2}$ is negative if $\frac{1}{3}C^{2}<\frac{1}{2}\left(  \lambda
_{0}^{2}+\lambda_{1}^{2}+\lambda_{2}^{2}+\lambda_{3}^{2}\right)  .$ This is
equivalent to the positivity of
\[
M:=3\left(  \lambda_{0}^{2}+\lambda_{1}^{2}+\lambda_{2}^{2}+\lambda_{3}%
^{2}\right)  -2\left(  \lambda_{0}+\lambda_{1}\right)  ^{2}=\lambda_{0}%
^{2}-4\lambda_{0}\lambda_{1}+\lambda_{1}^{2}+3\left(  \lambda_{2}^{2}%
+\lambda_{3}^{2}\right)  .
\]
When $\lambda_{0}\leq0\leq\lambda_{1}$then $-\lambda_{0}\lambda_{1}\geq0$ we
see that $M$ is positive if some $\lambda_{0},...,\lambda_{3}\neq0.$ In the
case $\lambda_{0}=\lambda_{1}=\lambda_{2}=\lambda_{3}=0$ we have the
polynomial case and the statement is well known. Next let us add the two
equations
\begin{align*}
M  &  =\left(  \lambda_{0}-2\lambda_{1}\right)  ^{2}-3\lambda_{1}^{2}+3\left(
\lambda_{2}^{2}+\lambda_{3}^{2}\right) \\
M  &  =\left(  \lambda_{1}-2\lambda_{0}\right)  ^{2}-3\lambda_{0}^{2}+3\left(
\lambda_{2}^{2}+\lambda_{3}^{2}\right)
\end{align*}
and we see that
\[
2M=\left(  \lambda_{0}-2\lambda_{1}\right)  ^{2}+\left(  \lambda_{1}%
-2\lambda_{0}\right)  ^{2}+3\left(  \lambda_{2}^{2}+\lambda_{3}^{2}%
-\lambda_{1}^{2}-\lambda_{0}^{2}\right)  .
\]
Thus $\lambda_{1}^{2}+\lambda_{0}^{2}<\lambda_{2}^{2}+\lambda_{3}^{2}$ implies
that $M$ is positive.
\end{proof}

\begin{remark}
Example \ref{ExampleNEW1} provides a function $\tau$ which has a local minimum
at $x=0.$
\end{remark}

Next we consider a class of examples which occur naturally in the study of polysplines.

\begin{proposition}
Let $a$ and $b$ be two non-zero complex numbers such that $a^{2}\neq b^{2}$
and $(\lambda_{_{0}},...,\lambda_{_{3}})=(a,-a,-b,b)$. Then
\begin{equation}
\rho_{0}\left(  x\right)  =\frac{(b-a)\sinh(a+b)x-(a+b)\sinh(b-a)x}{2ab\left(
b^{2}-a^{2}\right)  }. \label{represent_2}%
\end{equation}
If $a$ and $b$ are real then the functions $\rho\left(  x\right)  $ and
$-\rho\left(  -x\right)  $ are positive for all $x>0$.
\end{proposition}

\begin{proof}
It is easy to see that $\Phi_{(a,-a)}(x)=\frac{1}{a}\sinh ax$ and
$\Phi_{(b,-b)}(x)=\frac{1}{b}\sinh bx$ and
\begin{equation}
\Phi_{(a,-a,b,-b)}(x)=\frac{1}{b^{2}-a^{2}}\left(  \frac{1}{b}\sinh
bx-\frac{1}{a}\sinh ax\right)  . \label{eqphiex}%
\end{equation}
A computation shows that the numerator $\rho_{0}\left(  x\right)  $ of
$\rho\left(  x\right)  $ in (\ref{eqDefr}) is given by
\[
\rho_{0}\left(  x\right)  =\frac{b\sinh ax\cosh bx-a\cosh ax\sinh
bx}{ab\left(  b^{2}-a^{2}\right)  }%
\]
Since $\cosh(bx)\sinh(ax)=\frac{1}{2}\sinh(ax+bx)+\frac{1}{2}\sinh(ax-bx)$ we
obtain%
\begin{equation}
\rho_{0}\left(  x\right)  =\frac{(b-a)\sinh(a+b)x-(a+b)\sinh(b-a)x}{2ab\left(
b^{2}-a^{2}\right)  }. \label{eqrhoex}%
\end{equation}
If $a$ and $b$ are real then $\rho_{0}\left(  x\right)  $ is an exponential
polynomial with real frequencies $\pm\left(  a+b\right)  $ and $\pm\left(
b-a\right)  $, and it has therefore at most $3$ real zeros. Since $\rho_{0}$
has obviously a zero of order $3$ at $x=0$ it follows that $\rho\left(
x\right)  \neq0$ for all real $x\neq0,$ and it is now easy to see
that$\rho\left(  x\right)  $ and $-\rho\left(  -x\right)  $ are positive for
all $x>0.$
\end{proof}

Now we present the main result of this section:

\begin{theorem}
\label{TheoremSYMMETRICreal} Let $a,b\geq0$ and $\Lambda_{3}=(a,-a,-b,b)$.
Then the following estimate holds
\[
R_{_{j,j-1}}+R_{_{j,j+1}}\leq\frac{1}{2}R_{_{j,j}}%
\]
for all $t_{_{1}}<t_{_{2}}<...<t_{_{n}}$.
\end{theorem}

\begin{proof}
1. If $a=b=0$ this is known from the polynomial case. If $0<a=b$ this follows
from the results in \cite{KounchevRenderTsachevBIT}.

2. Now let us assume that $a,b>0$ and $a\neq0$. Since $A=\lambda
_{0}+...+\lambda_{3}=0$ Proposition \ref{PropositionMdeltaHalf} shows that
\[
\tau\left(  x\right)  =\frac{-\sigma\left(  -x\right)  }{\rho\left(  x\right)
}=\frac{\Phi_{\left(  -\lambda_{0},....,-\lambda_{3}\right)  }\left(
x\right)  }{\rho_{0}(x)}=\frac{\Phi_{(a,-a,b,-b)}(x)}{\rho_{0}(x)}.
\]
Further (\ref{eqphiex}) and (\ref{eqrhoex}) show that
\[
\tau\left(  x\right)  =\frac{2ab\left(  \frac{1}{b}\sinh bx-\frac{1}{a}\sinh
ax\right)  }{(b-a)\sinh(a+b)x-(a+b)\sinh(b-a)x}.
\]
Since $a+b>\max\left\{  a,b\right\}  $ we see that $\tau\left(  x\right)
\rightarrow0$ for $x\rightarrow\infty.$

3. Our assumptions imply that $\lambda_{0}+\lambda_{1}=0=-3\left(  \lambda
_{2}+\lambda_{3}\right)  .$ Then Proposition \ref{ThmMAX} shows that
$\tau\left(  x\right)  \leq\frac{1}{2}=\tau\left(  0\right)  $ for all $x$ in
a neighborhood of $0$ and $\tau\left(  x\right)  $ is strictly decreasing for
small $x>0.$ Let $\left(  0,\alpha\right)  $ be the maximal interval such that
$\tau^{\prime}\left(  x\right)  <0,$ so $\tau\left(  x\right)  $ is decreasing
on $\left(  0,\alpha\right)  .$ Assume that $\alpha<\infty.$ Then
$\tau^{\prime}\left(  \alpha\right)  =0,$ and $\tau$ is increasing on some
interval $\left(  \alpha,\alpha+\delta\right)  .$ Since $\tau\left(  x\right)
\rightarrow0$ for $x\rightarrow\infty$ we see that there exists $C>0$ such
that $0<x_{1}<\alpha<x_{2}<x_{3}$ with $\tau\left(  x_{j}\right)  =C$ for
$j=1,2,3.$ Since $\tau$ is obviously even, we have as well $\tau\left(
-x_{j}\right)  =C$ for $j=1,2,3$. We consider
\[
G\left(  x\right)  :=C\rho_{0}\left(  x\right)  -\Phi_{\left(  A-\lambda
_{0},....,A-\lambda_{3}\right)  }\left(  x\right)  .
\]
Then $G$ has a zero of order $3$ at $x=0$ and the zeros $x_{1},x_{3}%
,x_{3},-x_{1},-x_{2},-x_{3},$ so the function $G$ has at least $9$ zeros
(including multiplicities). But $G$ is an exponential polynomial with
frequencies $\pm\left(  a+b\right)  ,\pm\left(  b-a\right)  ,\pm b,\pm a,$
which has dimension $8,$ so $G$ can have at most $7$ zeros. This contradiction
shows that $\alpha=\infty,$ so $\tau\left(  x\right)  $ is decreasing for all
$x>0.$ Since $\tau$ is even we see that
\[
0\leq\tau\left(  x\right)  =\frac{-\sigma\left(  -x\right)  }{\rho\left(
x\right)  }\leq\frac{1}{2}.
\]
Now apply Theorem \ref{ThmDom}.

4. In the case that $\left(  0,0,b-b\right)  $ we have $\Phi_{(0,0)}(x)=x$,
and
\[
\text{ }\Phi_{(-b,b)}(x)=\frac{1}{b}\sinh bx\text{ and }\Phi_{(0,0,-b,b)}%
=\frac{1}{b^{2}}\left(  \frac{1}{b}\sinh bx-x\right)  .
\]
A computation shows that
\[
\rho_{0}\left(  x\right)  =\frac{1}{b^{3}}\left(  bx\cosh bx-\sinh bx\right)
.
\]
and $\rho_{0}$ an exponential polynomial which depends on the frequencies
$\left(  b,-b,b,-b\right)  .$ Further
\[
\tau\left(  x\right)  =\frac{\sinh bx-bx}{bx\cosh bx-\sinh bx}%
\]
is an even function and $\tau\left(  x\right)  \rightarrow0$ for
$x\rightarrow\infty.$ Now argue as in point $3.$

5. In the case $\left(  b,-b,0,0\right)  $ a computation shows that $\rho
_{0}\left(  x\right)  $ and $\tau\left(  x\right)  $ are identical to the
functions in item 4.
\end{proof}

The following example shows the surprising fact that the function $\tau$ may
have a local maximum at $x=0$ when $\Lambda_{3}$ is not symmetric.

\begin{example}
Let $\Lambda_{3}=\left(  -1,4,-2,1\right)  $. Then $\Phi_{\left(  -1,4\right)
}=\frac{e^{4x}-e^{-x}}{5}$ and $\Phi_{\left(  -2,1\right)  }\left(  x\right)
=\frac{e^{x}-e^{-2x}}{3}$ and it is easy to verify that%
\[
\Phi_{\left(  -1,4,-2,1\right)  }\left(  x\right)  =\frac{1}{10}e^{-x}%
-\frac{1}{18}e^{x}-\frac{1}{18}e^{-2x}+\frac{1}{90}e^{4x}\allowbreak.
\]
Then
\[
\tau\left(  x\right)  =\frac{-\sigma\left(  -x\right)  }{\rho_{0}\left(
x\right)  }=\frac{\left(  \frac{1}{10}e^{-x}-\frac{1}{18}e^{x}-\frac{1}%
{18}e^{-2x}+\frac{1}{90}e^{4x}\allowbreak\right)  e^{-2x}}{\frac{1}{15}%
e^{-2x}-\frac{1}{90}e^{3x}-\frac{1}{9}e^{-3x}+\frac{1}{30}e^{-5x}%
+\allowbreak\frac{1}{45}}\ .
\]
Then $\tau\left(  x\right)  $ has a local maximum at $x=0$ and $\tau\left(
x\right)  $ is bounded by $\frac{1}{2}.$ Note that $\Lambda_{3}$ is not
symmetric as the following graph shows.
\begin{center}
\fbox{\includegraphics[
height=2.8655in,
width=2.8067in,
keepaspectratio
]%
{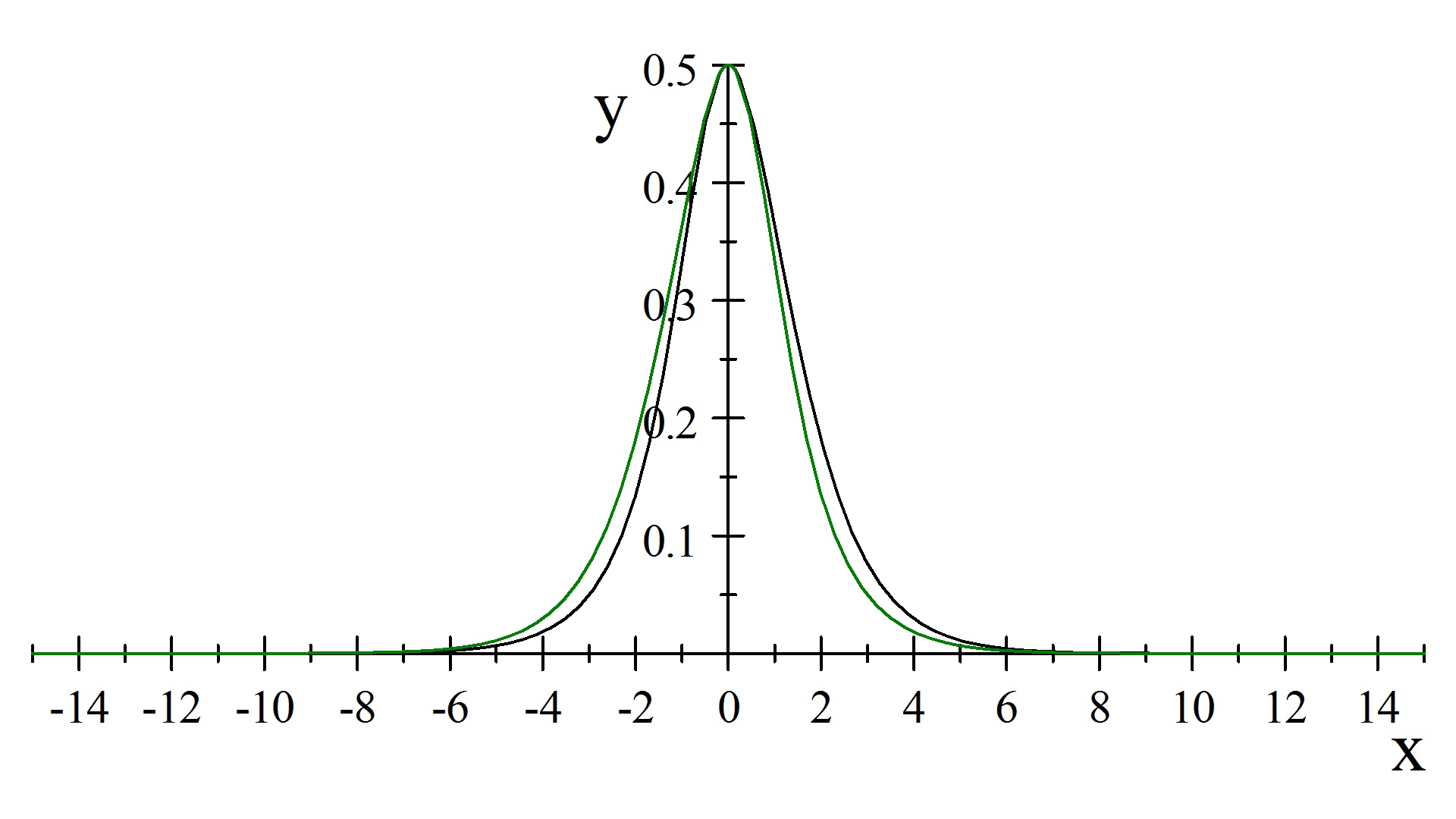}%
}\\
Fig. 3: $\varphi\left(  x\right)  $ in black, $\varphi\left(  -x\right)  $ in
green for $\left(  -1,4,-2,1\right)  $%
\end{center}

\end{example}

For the frequencies $0,0,b,-b$ one can form the vectors $\Lambda_{1}=\left(
0,0,b,-b\right)  $ and $\Lambda_{2}=\left(  b,-b,0,0\right)  ,$ and
$\Lambda_{3}=\left(  0,-b,0,b\right)  $ and $\Lambda_{4}=\left(
0,b,0,-b\right)  .$ For $\Lambda_{1}$ and $\Lambda_{2}$ we have proved strict
diagonal dominance with factor $\frac{1}{2}.$ Cases $\Lambda_{3}$ and
$\Lambda_{4}$ are identical if we replace $b$ by $-b.$ We show that also in
this case we have strict diagonal dominance:

\begin{proposition}
\label{PropEx1} If $\Lambda_{3}=(0,b,0,-b)$ for real $b\neq0$ then
$R_{_{j,j-1}}+R_{_{j,j+1}}<R_{_{j,j}}$ for all $t_{_{1}}<...<t_{_{n}}$, i.e.
 $R$ is strictly diagonally dominant.
\end{proposition}

\begin{proof}
For $b\neq0$ we have $\Phi_{(0,b)}(x)=\frac{e^{bx}-1}{b}$ and $\Phi
_{(0,-b)}(x)=\frac{e^{-bx}-1}{-b},$ and
\[
\Phi_{(0,b,0,-b)}=\frac{1}{b^{2}}\left(  \frac{1}{b}\sinh bx-x\right)  \text{
and }\sigma(x)=\frac{\sinh bx-bx}{b^{2}x\sinh bx}.
\]
A calculation shows that
\[
\rho_{0}\left(  x\right)  =\frac{\left(  \cosh bx-1\right)  \left(
e^{bx}-1\right)  -\left(  \sinh bx-bx\right)  e^{bx}}{b^{3}}=\frac{f\left(
bx\right)  }{b^{3}}%
\]
where $f\left(  x\right)  =\left(  \cosh x-1\right)  \left(  e^{x}-1\right)
-\left(  \sinh x-x\right)  e^{x}.$ Note that $f^{\prime}\left(  x\right)
=xe^{x}-\sinh x$ is positive for $x>0$ (note that $g=f^{\prime}$ in $E\left(
1,1,-1\right)  $ and it has a double zero at $x=0$, so it has no more zeros).
It follows that $\rho\left(  x\right)  >0$ for all $x>0$ and $-\rho\left(
-x\right)  >0$ for all $x>0.$ We apply Theorem \ref{ThmDom} and we have to
estimate
\[
\frac{-\sigma(-x)}{\rho(x)}=\frac{\Phi_{(0,b,0,-b)}\left(  x\right)  }%
{\rho_{_{0}}(x)}=\frac{b\sinh bx-bx}{\left(  \cosh bx-1\right)  \left(
e^{bx}-1\right)  -\left(  \sinh bx-bx\right)  e^{bx}}=g\left(  bx\right)
\]
for all $x\in\mathbb{R}\ $where
\[
g\left(  x\right)  =\frac{\sinh x-x}{\left(  \cosh x-1\right)  \left(
1-e^{-x}\right)  -\left(  \sinh x-x\right)  e^{-x}}.
\]
The behavior of the function $g$ is demonstrated in the following graph:
\begin{center}
\fbox{\includegraphics[
height=2.8596in,
width=2.8699in,
keepaspectratio
]%
{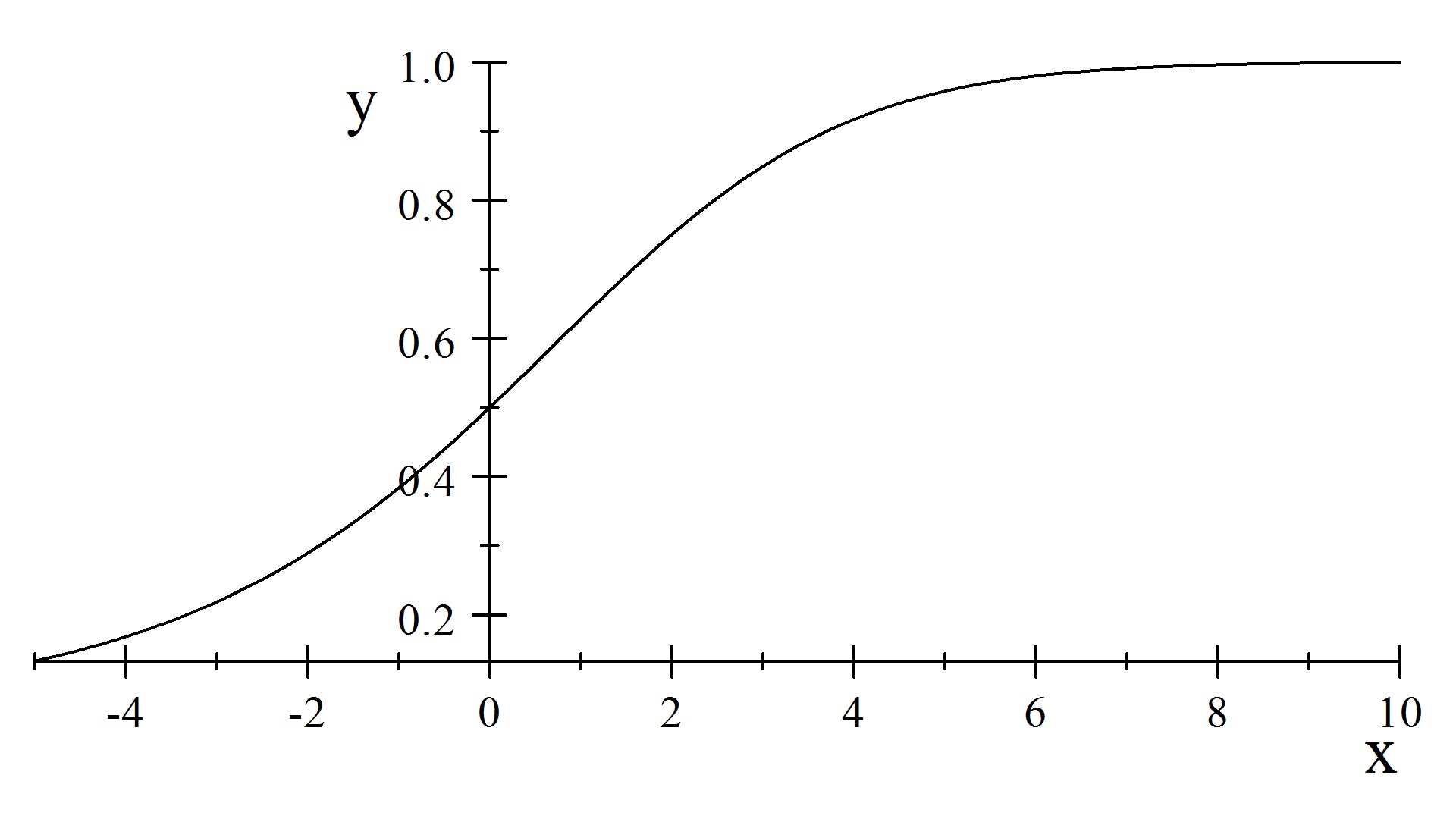}%
}\\
Fig. 4: Graph of $g\left(  x\right)  $%
\end{center}

\end{proof}

If we allow complex frequencies we still obtain strict diagonal dominance, but
again with weaker constants:

\begin{proposition}
Let $0<\delta\leq\pi$. Then for $\Lambda_{3}=(0,0,-i\beta,i\beta)$ with
$\beta>0$ we have
\[
R_{_{j,j-1}}+R_{_{j,j+1}}\leq\frac{\sin\delta-\delta}{\delta\cos\delta
-\sin\delta}R_{_{j,j}}<R_{_{j,j}}%
\]
for all $t_{_{1}}<...<t_{_{n}}$ such that $\beta(t_{_{j+1}}-t_{_{j}}%
)\leq\delta$ for $j=1,...,n-1.$
\end{proposition}

\begin{proof}
Clearly $\Phi_{(\lambda_{_{0}},\lambda_{_{1}})}(x)=x$ is positive for all
$x>0$, and
\[
\Phi_{(\lambda_{_{2}},\lambda_{_{3}})}(x)=\frac{1}{\beta}\sin\beta x
\]
is positive for all $x\in(0,\pi/\beta)$. Note that
\[
\Phi_{(0,0,-i\beta,i\beta)}(x)=\frac{1}{\beta^{3}}(\beta x-\sin\beta x)
\]
is positive for all $x>0$ since the function $x-\sin x$ is positive for all
$x>0.$ Then
\[
\frac{\sigma(x)}{-\rho(-x)}=\frac{\sin\beta x-\beta x}{\beta x\cos\beta
x-\sin\beta x}=f(\beta x)
\]
where
\begin{equation}
f(x):=\frac{\sin x-x}{x\cos x-\sin x} \label{func}%
\end{equation}
satisfies $\lim_{x\rightarrow0^{+}}f(x)=\frac{1}{2}$ which is obtained by
applying three times the l'Hospital's rule.

The function (\ref{func}) is positive in $(0,\pi)$ because $\sin x-x<0$ for
$x>0$, $x<\tan x$ with $\cos x>0$ on $(0,\pi/2)$ and $\cos x<0,\sin x>0$ on
$(\pi/2,\pi)$.

We next show that it is increasing in $(0,\pi)$. We first note that
$f^{\prime}(\pi)=2/\pi>0$, i.e. $f(\cdot)$ is increasing around $x=\pi$.
Hence, to establish that it is increasing in $(0,\pi)$, it suffices to show
that it is monotone in this interval. Aiming at contradiction, suppose that
there exist $x_{1}<x_{2}$ in $(0,\pi)$ such that $\frac{\sin x_{j}-x_{j}%
}{x_{j}\cos x_{j}-\sin x_{j}}=C>0$ for $j=1,2.$ Then
\[
F(x)=\sin x-x-C(x\cos x-\sin x)
\]
vanishes at $x_{1}$ and $x_{2},$ and clearly $F$ has a zero of order $3$ at
$0.$ By the theorem of Rolle
\[
G(x)=\frac{d^{2}}{dx^{2}}\left(  (\sin x-x)-C(x\cos x-\sin x)\right)  =C\sin
x-\sin x+Cx\cos x
\]
has at least two zeros, say $y_{1}<y_{2}$, in $(0,\pi)$. Define $H(x)=x\cot
x$. Then
\begin{equation}
H(y_{j})=\frac{1-C}{C} \label{contra}%
\end{equation}
for $j=1,2$. But the function $H(x)=x\cot x$ is decreasing on $(0,\pi)$
because
\[
H^{\prime}(x)=\frac{\sin x\cos x-x}{\sin^{^{2}}x}=\frac{\frac{1}{2}\sin
2x-x}{\sin^{^{2}}x}.
\]
The numerator is a function vanishing at $x=0$ and having derivative
$\cos2x-1<0$, i.e. the numerator is negative in $(0,\pi)$ which yields
$H^{\prime}(x)<0$ in $(0,\pi)$, i.e. $H(\cdot)$ is decreasing in this
interval. But this contradicts (\ref{contra}). Hence, $H(\cdot)$ is increasing
in $(0,\pi)$. This yields
\[
\frac{\sigma(x)}{-\rho(-x)}\leq\frac{\sin\delta-\delta}{\delta\cos\delta
-\sin\delta}=:M_{\delta}%
\]
for all $x\in\left[  0,\delta/\beta\right]  $. Note that $M_{\pi}=1$.
\end{proof}

\section{Positivity properties of the matrix $R$}

In this section we shall show that the entries in (\ref{eqRjj}) and
(\ref{eqRjj1}) are positive numbers provided that $\lambda_{0},...,\lambda
_{3}$ are real numbers. In combination with the fact that a strictly
diagonally dominant matrix with non-negative entries is positive definite, one
deduce from these results sufficient criteria for the positive-definiteness of
the matrix $R$ which we will not formulate explicitly here.

For a simple proof of the positivity of entries in (\ref{eqRjj}) and
(\ref{eqRjj1}) the next Theorem \ref{ThmMainrho} will be very useful. It
states the crucial fact that $\rho_{0},$ the numerator of $\rho\left(
x\right)  $ defined in (\ref{eqDefrhonull}), is an exponential polynomial
which depends only on five frequencies, so $\rho_{0}$ is in the vector space
\begin{equation}
E\left(  \lambda_{0}+\lambda_{2},\quad\lambda_{0}+\lambda_{3},\quad\lambda
_{1}+\lambda_{2},\quad\lambda_{1}+\lambda_{3},\quad\lambda_{0}+\lambda
_{1}\right)  . \label{eqES}%
\end{equation}

This is surprising since the definition of $\rho_{0}\left(  x\right)  $ is
based on the products $\Phi_{\left(  \lambda_{0},....,\lambda_{3}\right)
}\Phi_{\left(  \lambda_{0},\lambda_{1}\right)  }^{\prime}$ and $\Phi_{\left(
\lambda_{0},....,\lambda_{3}\right)  }^{\prime}\Phi_{\left(  \lambda
_{0},\lambda_{1}\right)  }$ which are clearly exponential polynomials with
frequencies $\lambda_{j}+\lambda_{k}$ where $j=0,...,3$ and $k=0,1$. Thus it
is an obvious fact that $\rho_{0}$ can be written as a exponential polynomial
with $8$ frequencies.

\begin{theorem}
\label{ThmMainrho}Let $\left(  \lambda_{0},\lambda_{1},\lambda_{2},\lambda
_{3}\right)  \in\mathbb{C}^{4}$ and define $D_{\lambda}f=f^{\prime}-\lambda
f.$ Then
\[
D_{\lambda_{1}+\lambda_{0}}\rho_{0}\left(  x\right)  =\Phi_{\left(
\lambda_{0},\lambda_{1}\right)  }\left(  x\right)  \Phi_{\left(  \lambda
_{2},\lambda_{3}\right)  }\left(  x\right)  .
\]
Moreover, $\rho_{0}$ is an exponential polynomial contained in (\ref{eqES})
satisfying the initial conditions $\rho_{0}^{\left(  j\right)  }\left(
0\right)  =0$ for $j=0,1,2$ and%
\[
\rho_{0}^{\left(  3\right)  }\left(  0\right)  =2\text{ and }\rho_{0}^{\left(
4\right)  }\left(  0\right)  =5\lambda_{0}+5\lambda_{1}+3\lambda_{2}%
+3\lambda_{3}.
\]

\end{theorem}

\begin{proof}
It is trivial to see that
\[
\rho_{0}=D_{\lambda_{0}}\Phi_{(\lambda_{_{0}},...,\lambda_{_{3}})}\cdot
\Phi_{(\lambda_{_{0}},\lambda_{_{1}})}-\Phi_{(\lambda_{_{0}},...,\lambda
_{_{3}})}\cdot D_{\lambda_{0}}\Phi_{(\lambda_{_{0}},\lambda_{_{1}})}.
\]
Using that
\begin{equation}
D_{\left(  \lambda+\mu\right)  }\left(  fg\right)  =D_{\lambda}f\cdot g+f\cdot
D_{\mu} \label{eqDformula}%
\end{equation}
it follows that
\begin{align*}
D_{\lambda_{1}+\lambda_{0}}\rho_{0}  &  =D_{\lambda_{1}}D_{\lambda_{0}}%
\Phi_{(\lambda_{_{0}},...,\lambda_{_{3}})}\cdot\Phi_{(\lambda_{_{0}}%
,\lambda_{_{1}})}\ +D_{\lambda_{0}}\Phi_{(\lambda_{_{0}},...,\lambda_{_{3}}%
)}\cdot D_{\lambda_{0}}\Phi_{(\lambda_{_{0}},\lambda_{_{1}})}\\
&  -D_{\lambda_{0}}\Phi_{(\lambda_{_{0}},...,\lambda_{_{3}})}\cdot
D_{\lambda_{0}}\Phi_{(\lambda_{_{0}},\lambda_{_{1}})}-\Phi_{(\lambda_{_{0}%
},...,\lambda_{_{3}})}\cdot D_{\lambda_{1}}D_{\lambda_{0}}\Phi_{(\lambda
_{_{0}},\lambda_{_{1}})}.
\end{align*}
The last summand is zero, the second and third summand cancel. Since
$D_{\lambda_{1}}D_{\lambda_{0}}\Phi_{(\lambda_{_{0}},...,\lambda_{_{3}})}%
=\Phi_{(\lambda_{2},\lambda_{_{3}})}$ by (\ref{rec0}) we conclude that
\[
D_{\lambda_{1}+\lambda_{0}}\rho_{0}\left(  x\right)  =\Phi_{(\lambda
_{2},\lambda_{_{3}})}(x)\Phi_{(\lambda_{_{0}},\lambda_{_{1}})}(x).
\]
We look now at the derivatives of $\rho_{0}.$ We know that
\[
D_{\lambda_{1}+\lambda_{0}}\rho_{0}=\rho_{0}^{\prime}-\left(  \lambda
_{0}+\lambda_{1}\right)  \rho_{0}=\Phi_{\left(  \lambda_{0},\lambda
_{1}\right)  }\Phi_{\left(  \lambda_{2},\lambda_{3}\right)  }.
\]
Further we know that $\rho_{0}^{\left(  j\right)  }\left(  0\right)  =0$ for
$j=0,1,2.$ Then%
\[
\rho_{0}^{\prime\prime\prime}=\left(  \lambda_{0}+\lambda_{1}\right)  \rho
_{0}^{\prime\prime}+\Phi_{\left(  \lambda_{0},\lambda_{1}\right)  }%
^{\prime\prime}\Phi_{\left(  \lambda_{2},\lambda_{3}\right)  }+2\Phi_{\left(
\lambda_{0},\lambda_{1}\right)  }^{\prime}\Phi_{\left(  \lambda_{2}%
,\lambda_{3}\right)  }^{\prime}+\Phi_{\left(  \lambda_{0},\lambda_{1}\right)
}\Phi_{\left(  \lambda_{2},\lambda_{3}\right)  }^{\prime\prime}.
\]
Since $\rho_{0}^{\prime\prime}\left(  0\right)  =0$ and $\Phi_{\left(
\lambda_{0},\lambda_{1}\right)  }\left(  0\right)  =\Phi_{\left(  \lambda
_{2},\lambda_{3}\right)  }=0$ we have
\[
\rho_{0}^{\prime\prime\prime}\left(  0\right)  =2\Phi_{\left(  \lambda
_{0},\lambda_{1}\right)  }^{\prime}\left(  0\right)  \Phi_{\left(  \lambda
_{2},\lambda_{3}\right)  }^{\prime}\left(  0\right)  =2
\]
Further by Leibniz's rule
\[
\rho_{0}^{\left(  4\right)  }=\left(  \lambda_{0}+\lambda_{1}\right)  \rho
_{0}^{\left(  3\right)  }+\Phi_{\left(  \lambda_{0},\lambda_{1}\right)
}^{\prime\prime\prime}\Phi_{\left(  \lambda_{2},\lambda_{3}\right)  }%
+3\Phi_{\left(  \lambda_{0},\lambda_{1}\right)  }^{\prime\prime}\Phi_{\left(
\lambda_{2},\lambda_{3}\right)  }^{\prime}+\Phi_{\left(  \lambda_{0}%
,\lambda_{1}\right)  }^{\prime}\Phi_{\left(  \lambda_{2},\lambda_{3}\right)
}^{\prime\prime}+\Phi_{\left(  \lambda_{0},\lambda_{1}\right)  }\Phi_{\left(
\lambda_{2},\lambda_{3}\right)  }^{\prime\prime\prime}.
\]
and we obtain
\[
\rho_{0}^{\left(  4\right)  }\left(  0\right)  =2\left(  \lambda_{0}%
+\lambda_{1}\right)  +3\left(  \lambda_{0}+\lambda_{1}\right)  +3\left(
\lambda_{2}+\lambda_{3}\right)  =5\left(  \lambda_{0}+\lambda_{1}\right)
+3\left(  \lambda_{2}+\lambda_{3}\right)  .
\]

\end{proof}

\begin{theorem}
\label{positivity} Suppose that $\lambda_{0},...,\lambda_{3}$ are real
numbers. Then the values $\sigma\left(  x\right)  $ and $-\sigma\left(
-x\right)  $, and $\rho\left(  x\right)  $ and $-\rho\left(  -x\right)  $ are
positive for all $x>0,$
\end{theorem}

\begin{proof}
We know that $\Phi_{\Lambda_{3}}(x)>0$ for all $x>0$ whenever $\lambda
_{0},...,\lambda_{3}$ are real, and therefore
\[
\sigma(x)=\frac{\Phi_{(\lambda_{_{0}},...,\lambda_{_{3}})}(x)}{\Phi
_{(\lambda_{_{0}},\lambda_{_{1}})}(x)\Phi_{(\lambda_{_{2}},\lambda_{_{3}}%
)}(x)}>0
\]
for all $x>0.$ Further on, for $x>0$ we have
\[
\sigma(-x)=\frac{\Phi_{(\lambda_{_{0}},...,\lambda_{_{3}})}(-x)}%
{\Phi_{(\lambda_{_{0}},\lambda_{_{1}})}(-x)\Phi_{(\lambda_{_{2}},\lambda
_{_{3}})}(-x)}=\frac{(-1)^{3}\Phi_{(-\lambda_{_{0}},...,-\lambda_{_{3}})}%
(x)}{\Phi_{(-\lambda_{_{0}},-\lambda_{_{1}})}(x)\Phi_{(-\lambda_{_{2}%
},-\lambda_{_{3}})}(x)}<0.
\]
Note that $\rho_{0}$ is a real-valued function. By Theorem \ref{ThmMainrho} we
have
\[
\frac{d}{dx}\left(  \rho_{0}\left(  x\right)  e^{\left(  \lambda_{0}%
+\lambda_{1}\right)  x}\right)  =e^{\left(  \lambda_{0}+\lambda_{1}\right)
x}D_{\lambda_{0}+\lambda_{1}}\left(  \rho_{0}\left(  x\right)  \right)
=\Phi_{(\lambda_{_{0}},\lambda_{_{1}})}(x)\Phi_{(\lambda_{_{2}},\lambda_{_{3}%
})}(x).
\]
Note that $\Phi_{(\lambda_{_{0}},\lambda_{_{1}})}(x)$ and $\Phi_{(\lambda
_{_{2}},\lambda_{_{3}})}(x)$ are positive for $x>0$, and negative for $x<0.$
Thus $\Phi_{(\lambda_{_{0}},\lambda_{_{1}})}(x)\Phi_{(\lambda_{_{2}}%
,\lambda_{_{3}})}(x)>0$ for all $x\neq0.$ Hence the function $\rho_{0}\left(
x\right)  e^{\left(  \lambda_{0}+\lambda_{1}\right)  x}$ is increasing for
$x>0$ and since $\rho_{0}\left(  0\right)  =0$ it follows that $\rho
_{0}\left(  x\right)  >0$ for all $x>0,$ and $\rho_{0}\left(  x\right)  <0$
for all $x<0.$
\end{proof}

\begin{corollary}
Suppose that $\lambda_{_{0}},...,\lambda_{_{3}}$ are real numbers. Then the
entries $R_{_{j,j}}$, as well as $R_{_{j,j-1}}$ and $R_{_{j,j+1}}$ are
positive numbers for all $t_{_{1}}<...<t_{_{n}}$.
\end{corollary}

\section{ Symmetry properties}

Since in the polynomial case the matrix R is symmetric, it is natural to
investigate about the symmetry in the more general case. In the present
section we study some symmetry properties of the matrix $R$ and of the
functions $\rho$ and $\sigma.$

\begin{theorem}
Let $\Lambda_{3}=(\lambda_{_{0}},...,\lambda_{_{3}})$ be given. Then the
following statements are equivalent:

a) $\Lambda_{3}$ is symmetric.

b) $\sigma(x)$ is odd.

c) The matrix $R$ for is symmetric for any choices of interpolation points
$t_{_{1}}<...<t_{n}.$

d) $\Phi_{\left(  \lambda_{0},\lambda_{1},\lambda_{2},\lambda_{3}\right)
}\left(  x\right)  $ is odd.
\end{theorem}

\begin{proof}
For $a)\rightarrow b)$ assume that $\Lambda_{3}$ is symmetric. Then
$\sigma(x)$ is odd since by Proposition \ref{PropSym}
\[
\sigma(-x)=\frac{\Phi_{(\lambda_{_{0}},...,\lambda_{_{3}})}(-x)}%
{\Phi_{(\lambda_{_{0}},\lambda_{_{1}})}(-x)\Phi_{(\lambda_{_{2}},\lambda
_{_{3}})}(-x)}=\frac{-\Phi_{(\lambda_{_{0}},...,\lambda_{_{3}})}(x)}%
{\Phi_{(\lambda_{_{0}},\lambda_{_{1}})}(x)\Phi_{(\lambda_{_{2}},\lambda_{_{3}%
})}(x)}=-\sigma(x).
\]
For $b)\rightarrow c)$ assume that $\sigma$ is odd. Then $R$ is symmetric
since by (\ref{eqRjj1})
\[
R_{_{j,j+1}}=\sigma(t_{_{j+1}}-t_{_{j}})=-\sigma\left(  -(t_{_{j+1}}-t_{_{j}%
})\right)  =R_{_{j+1,j}}.
\]
For $c)\rightarrow d)$ assume that $R$ is symmetric for any choices of
interpolation points $t_{_{1}}<...<t_{_{n}}.$ It follows that that the
function $\sigma(x)$ is odd. According to (\ref{eqPsiEven}) we can write
\[
\sigma(x)=\frac{\Phi_{(\lambda_{_{0}},...,\lambda_{_{3}})}(x)e^{^{-\left(
\lambda_{_{0}}+\lambda_{_{1}}+\lambda_{_{2}}+\lambda_{_{3}}\right)  x/2}}%
}{\psi_{(\lambda_{_{0}},\lambda_{_{1}})}(x)\psi_{(\lambda_{_{2}},\lambda
_{_{3}})}(x)}%
\]
where $\psi_{(\lambda_{_{0}},\lambda_{_{1}})}(x)\psi_{(\lambda_{_{2}}%
,\lambda_{_{3}})}(x)$ is even. Since $\sigma$ is odd it follows that
\[
\varphi(x)=\Phi_{(\lambda_{_{0}},...,\lambda_{_{3}})}(x)e^{^{-\left(
\lambda_{_{0}}+\lambda_{_{1}}+\lambda_{_{2}}+\lambda_{_{3}}\right)  x/2}}%
\]
is odd. Note that $\varphi(\cdot)$ has a zero of order $3$ at $x=0$. If
$\varphi(\cdot)$ is odd then the even derivatives are zero, so $\tau
^{^{\left(  4\right)  }}(0)=0$. According to Proposition 4 in \cite{KoRe13} we
have
\[
\Phi_{(\lambda_{_{0}},...,\lambda_{_{3}})}^{^{\left(  4\right)  }}%
(0)=\lambda_{_{0}}+\lambda_{_{1}}+\lambda_{_{2}}+\lambda_{_{3}}.
\]
Using Leibniz's formula we obtain
\begin{align*}
\varphi^{^{\left(  4\right)  }}(0)  &  =\sum_{k=0}^{4}\binom{4}{k}%
\Phi_{(\lambda_{_{0}},...,\lambda_{_{3}})}^{^{\left(  k\right)  }}(0)\left(
-\left(  \lambda_{_{0}}+\lambda_{_{1}}+\lambda_{_{2}}+\lambda_{_{3}}\right)
/2\right)  ^{^{\left(  4-k\right)  }}\\
&  =4\left(  -\left(  \lambda_{_{0}}+\lambda_{_{1}}+\lambda_{_{2}}%
+\lambda_{_{3}}\right)  /2\right)  ^{^{\left(  4-3\right)  }}+\Phi
_{(\lambda_{_{0}},...,\lambda_{_{3}})}^{^{\left(  4\right)  }}(0)\\
&  =\left(  \lambda_{_{0}}+\lambda_{_{1}}+\lambda_{_{2}}+\lambda_{_{3}%
}\right)  \left(  -2+1\right)  .
\end{align*}
Since $\varphi^{^{\left(  4\right)  }}(0)=0$ it follows that $\lambda_{_{0}%
}+\lambda_{_{1}}+\lambda_{_{2}}+\lambda_{_{3}}=0$. Hence $\Phi_{(\lambda
_{_{0}},...,\lambda_{_{3}})}(\cdot)\equiv\tau(\cdot)$ is odd.

For $d)\rightarrow a)$ assume that $\Phi_{(\lambda_{_{0}},...,\lambda_{_{3}}%
)}$ is odd. Write
\[
\Phi_{(\lambda_{_{0}},...,\lambda_{_{3}})}(x)=\sum_{k=0}^{l}P_{_{k}%
}(x)e^{^{\lambda_{k}x}}.
\]
where $l$ is an integer between $0$ and $3$, $P_{_{k}}(\cdot)$ are nontrivial
polynomials and $\lambda_{k}\neq\lambda_{j}$ for $k\neq j$. The fact that
$\Phi_{(\lambda_{_{0}},...,\lambda_{_{3}})}(\cdot)$ is odd yields
\[
\sum_{k=0}^{l}P_{_{k}}(x)e^{^{\lambda_{k}x}}+\sum_{k=0}^{l}P_{_{k}%
}(-x)e^{^{-\lambda_{k}x}}\equiv0.
\]
From here one easily obtains that the $\lambda$-s have to be pairwise opposite
to each other.
\end{proof}

It would be interesting to know when $\rho\left(  x\right)  $ is odd.

\begin{theorem}
Let $\Lambda_{3}=(\lambda_{_{0}},...,\lambda_{_{3}})$ be given. Then the
function $\rho\left(  x\right)  $ is odd if and only if $\lambda_{0}%
+\lambda_{1}=\lambda_{2}+\lambda_{3}.$
\end{theorem}

\begin{proof}
Corollary \ref{CorOdd} shows the necessity. For the converse, we see at first
by Proposition \ref{PropSym} that
\[
\rho\left(  -x\right)  =\frac{\rho_{0}\left(  -x\right)  }{\Phi_{(\lambda
_{_{0}},\lambda_{_{1}})}(-x)\Phi_{(\lambda_{_{2}},\lambda_{_{3}})}(-x)}%
=\frac{\rho_{0}\left(  -x\right)  e^{(\lambda_{_{0}}+\lambda_{_{1}}%
+\lambda_{_{2}}+\lambda_{_{3}})x}}{\Phi_{\left(  \lambda_{_{0}},\lambda_{_{1}%
}\right)  }(x)\Phi_{\left(  \lambda_{_{2}},\lambda_{_{3}}\right)  }(x)}\
\]
Thus $\rho\left(  x\right)  $ is odd if and only if
\[
\rho_{0}\left(  -x\right)  e^{(\lambda_{_{0}}+\lambda_{_{1}}+\lambda_{_{2}%
}+\lambda_{_{3}})x}=-\rho_{0}\left(  x\right)  .
\]
From Theorem \ref{ThmMainrho} it is easy to see that the exponential
polynomial
\[
g\left(  x\right)  =\rho_{0}\left(  -x\right)  e^{(\lambda_{_{0}}%
+\lambda_{_{1}}+\lambda_{_{2}}+\lambda_{_{3}})x}.
\]
in
\[
E\left(  \lambda_{0}+\lambda_{2},\quad\lambda_{0}+\lambda_{3},\quad\lambda
_{1}+\lambda_{2},\quad\lambda_{1}+\lambda_{3},\quad\lambda_{2}+\lambda
_{3}\right)
\]
where the last exponent stems from the calculation $-\left(  \lambda
_{0}+\lambda_{1}\right)  +\lambda_{_{0}}+\lambda_{_{1}}+\lambda_{_{2}}%
+\lambda_{_{3}}=\lambda_{2}+\lambda_{3}.$ Our assumption ensures that
$g\left(  x\right)  $ and $h\left(  x\right)  :=-\rho_{0}\left(  x\right)  $
are in the same exponential space (\ref{eqES}). In order to prove that $g$ and
$h$ are identical it suffices to show that $g^{\left(  j\right)  }\left(
0\right)  =h^{\left(  j\right)  }\left(  0\right)  $ for $j=0,1,2,3,4.$
Clearly this is true for $j=0,1,2,$ and clearly
\[
g^{\prime\prime\prime}\left(  0\right)  =\left(  -1\right)  ^{3}\rho
_{0}^{\prime\prime\prime}\left(  0\right)  =-2=h^{\prime\prime\prime}\left(
0\right)  .
\]
Further we see that
\begin{align*}
g^{\left(  4\right)  }\left(  0\right)   &  =\rho_{0}^{\left(  4\right)
}\left(  0\right)  +\binom{4}{3}\left(  -1\right)  2\rho_{0}^{\left(
3\right)  }\left(  0\right)  \left(  \lambda_{0}+\lambda_{1}+\lambda
_{2}+\lambda_{3}\right) \\
&  =-3\lambda_{0}-3\lambda_{1}-5\lambda_{2}-5\lambda_{3}%
\end{align*}
where we used that $\rho_{0}^{\left(  4\right)  }\left(  0\right)
=5\lambda_{0}+5\lambda_{1}+3\lambda_{2}+3\lambda_{3},$ and we see that
$g^{\left(  4\right)  }\left(  0\right)  =-\rho_{0}^{\left(  4\right)
}\left(  0\right)  =h^{\left(  4\right)  }\left(  0\right)  .$ The proof is accomplished.
\end{proof}

\bigskip
ACKNOWLEDGEMENTS

The work of H. Render and Ts. Tsachev was funded under project
KP-06-N32-8, while the work of O. Kounchev was funded under project KP-06N42-2 with Bulgarian
NSF.
The  work of the first-named author has been partially supported by Grant No BG05M2OP001-1.001-0003, financed by the Science and Education for Smart Growth Operational Program  (2014-2020)  and co-financed by the European Union through the European structural and Investment funds.

\end{document}